\documentclass[12pt,reqno]{amsart}

\usepackage{amsfonts,amsthm,amsmath,amssymb}
\makeatletter
\@ifundefined{subjclassname@2020}{%
  \@namedef{subjclassname@2020}{\textup{2020} Mathematics Subject Classification}%
}{}
\makeatother
\usepackage{hyperref}
\hypersetup{
    pdffitwindow=false,
    pdfstartview={FitH},
    colorlinks=true,
    citecolor=blue,
    linkcolor=red,
    urlcolor=blue
}

\usepackage{indentfirst}
\usepackage{cite}
\usepackage{enumerate}
\usepackage{bm}
\usepackage{txfonts}
\usepackage{booktabs}

\newtheorem{theorem}{Theorem}[section]
\newtheorem{lemma}[theorem]{Lemma}
\newtheorem{proposition}[theorem]{Proposition}

\theoremstyle{definition}
\newtheorem{remark}{Remark}[section]

\DeclareMathOperator{\divg}{div}
\DeclareMathOperator{\supp}{supp}
\DeclareMathOperator{\dist}{dist}
\newcommand{\R}{\mathbb{R}}

\newcommand{\p}{\partial}

\begin{document}
	\title[Stationary Navier-Stokes Solutions with Transverse bound]
	{Rigidity of Stationary Navier-Stokes Solutions with Transverse bound
		in Dimensions Five and Higher}

	\author{Changfeng Gui}
	\address{Department of Mathematics,  Faculty of Science, University of Macau, Taipa, Macao}
	\email{Changfenggui@um.edu.mo}

	\author{Hao Liu}
	\address{Department of Mathematics, Faculty of Science,  University of Macau, Taipa, Macao}
	\email{haoliu@um.edu.mo}

	\author{Yun Wang}
	\address{School of Mathematical Sciences, Soochow University, Suzhou, China}
	\email{ywang3@suda.edu.cn}

	\author{Chunjing Xie}
	\address{School of Mathematical Sciences, 	Ministry of Education Key Laboratory of Scientific and Engineering Computing,
				and CMA-Shanghai, Shanghai Jiao Tong University, 800 Dongchuan Road, Shanghai, China}
	\email{cjxie@sjtu.edu.cn}



	\begin{abstract}
For $n\geq5$, let $(u,p)$ be a smooth solution of the stationary
incompressible Navier-Stokes equations in
$\R^n\setminus\{x'=0\}$, where $\{x'=0\}$ is the $x_n$-axis.  We prove
that the scale-invariant transverse bound
$|u(x)|\leq C |x'|^{-1}$ forces $u\equiv0$.  This result relies on a
quantitative local regularity theorem, which shows that the transverse bound
condition yields uniform control of $u$ and $\nabla u$ in a smaller ball and
removes the possible singularity along the axis.
The proof combines weak extension across the axis, approximate Green
functions for an adjoint drift operator, a one-sided bound for the total head
pressure, a localized Frehse-R\r{u}\v{z}i\v{c}ka weighted estimate, and a
finite Stokes-Morrey bootstrap.
Applying the local theorem at arbitrarily large scales yields whole-space
rigidity and removability of line singularities. This together with the
asymptotic expansions on the exterior domains in the  case $|u(x)|\leq C |x|^{-1}$ in \cite{BangGuiLiuWangXie},  yields the same expansion in exterior domains under the condition $|u(x)|\leq C |x'|^{-1}$.
\end{abstract}

    \maketitle


	\section{Introduction and main results}

	\subsection{Background}

	We consider the stationary incompressible Navier-Stokes equations in $\R^n$
	\begin{equation}\label{NS}
		\begin{cases}
			-\Delta u + (u\cdot\nabla)u + \nabla p = 0,\\[2pt]
			\divg u = 0,
		\end{cases}
	\end{equation}
	where $u=(u_1,\dots,u_n)$ is the velocity and $p$ is the pressure.
	System \eqref{NS} possesses a natural \emph{scaling invariance}: if
	$(u,p)$ is a solution, then $(\lambda u(\lambda x),\lambda^2 p(\lambda x))$ is also a solution for
	every $\lambda>0$.  This makes it natural to study scale-invariant solutions or, more generally, solutions satisfying scale-invariant bounds.
	Of particular importance are solutions satisfying the scale-invariant bound
	\begin{equation}\label{isotropic}
		|u(x)|\leq\frac{C}{|x|},\quad x\in\mathbb{R}^n\setminus\{0\},
	\end{equation}
	whose class we denote by $\mathcal{L}^\infty_r$. It contains the 
	rotated self-similar and rotated discretely self-similar solutions in
	$\R^n\setminus\{0\}$; see Tsai~\cite{Tsai} and the references therein for detailed discussions of these types of solutions.

	In dimensions $n=2,3$, there exist nontrivial explicit self-similar solutions.
	In $n=2$, these are Jeffery-Hamel solutions~\cite{Jeffery,Hamel,Sverak}, see
	Guillod-Wittwer~\cite{GuillodWittwer15} for its extension to 	rotated self-similar solutions.
	In $n=3$, Landau~\cite{Landau} discovered $(-1)$-homogeneous axisymmetric solutions
	expressed via the stream function $\psi(r,\theta)=r f(\theta)$ with
	$f(\theta)=\frac{2\sin^2\theta}{a-\cos\theta}$ ($|a|>1$), associated with a point force at the origin;
	\v{S}ver\'{a}k~\cite{Sverak} proved these are the \emph{only} $(-1)$-homogeneous solutions
	in $\R^3\setminus\{0\}$.
	Hence the bound $O(1/|x|)$ is sharp in low dimensions: any rigidity result
	claiming $u\equiv 0$ under~\eqref{isotropic} must fail for $n=2,3$.

	In dimensions $n\geq4$, every smooth $(-1)$-homogeneous solution in
	$\R^n\setminus\{0\}$ is trivial; see
	\v{S}ver\'{a}k~\cite{Sverak} and Tsai~\cite{Tsai}.
The extension of rigidity from self-similar solutions to the full isotropic
class~\eqref{isotropic} was posed by \v{S}ver\'ak~\cite[p.~210]{Sverak}.
Bang, Gui, Liu, Wang, and Xie~\cite{BangGuiLiuWangXie} recently resolved
this isotropic problem for every $n\geq4$. The main result in \cite{BangGuiLiuWangXie}
can be stated as follows.
\begin{theorem}[{\cite[ Theorem 1.1]{BangGuiLiuWangXie}}]\label{thm:BGLWX}
 Let $n\geq4$ and let $u\in C^\infty(\mathbb{R}^n\setminus\{0\})$ be a solution
 to~\eqref{NS} in $\mathbb{R}^n\setminus\{0\}$ satisfying
 $|u(x)|\leq C/|x|$ for some $C>0$. Then $u\equiv0$.
\end{theorem}

The rigidity problem is closely related to the classical question of
removable singularities for stationary Navier--Stokes flows.
The removability of isolated singularities has been studied extensively;
see, among others, Dyer--Edmunds~\cite{DyerEdmunds70},
Shapiro~\cite{Shapiro74}, and Choe--Kim~\cite{ChoeKim00}.
In particular, Kim and Kozono~\cite{KimKozono06} proved that, for
$n\geq3$, an isolated singularity of a smooth stationary solution is
removable provided
\[
    u\in L^n
    \qquad\text{or}\qquad
    |u(x)|=o(|x|^{-1})
    \quad\text{as }x\to0.
\]
More recently, as an application of Theorem~\ref{thm:BGLWX},
Bang--Gui--Liu--Wang--Xie~\cite{BangGuiLiuWangXie} showed that
in dimensions $n\geq4$ the critical pointwise bound
\[
    |u(x)|=O(|x|^{-1})
\]
already implies removability of an isolated singularity.
Recently, Chen, Wang, Wang, Yang, and Yu~\cite{Chen26}
proved the rigidity of five-dimensional steady Navier--Stokes flows
under the critical Morrey condition
\[
    \sup_{R>0}R^{-2}\int_{B_R}|u|^3\,dx<\infty.
\]
This generalizes the result of
\cite{BangGuiLiuWangXie} in dimension five.
They also obtained corresponding removable-singularity criteria for
isolated singularities under bounded scale-invariant energy or cubic
Morrey control. 

The key quantity used in their approach and in the present work is the
total head pressure
\begin{equation}\label{Hdef}
 H=\frac12|u|^2+p.
\end{equation}
For smooth solutions it satisfies
\begin{equation}\label{Heq}
 -\Delta H+u\cdot\nabla H
 =-\frac12\sum_{i,j}|\partial_i u_j-\partial_j u_i|^2\leq0.
\end{equation}
The one-sided structure of this equation has played a central role
in the regularity theory developed by Frehse-R\r{u}\v{z}i\v{c}ka~\cite{FrehseRuzicka94,
FrehseRuzicka96}, Struwe~\cite{Struwe95}, and Li-Yang~\cite{LiYang22} for the existence of regular solutions.

\subsection{Main results}
	In this paper, we consider solutions that may be singular along a prescribed
axis.  Without loss of generality, we take this axis to be the $x_n$-axis and
consider solutions in
	$\mathbb{R}^n\setminus\{x'=0\}$, where $x=(x',x_n)$ and $\{x'=0\}$ denotes the
	$x_n$-axis.  We impose only the following \emph{transverse} bound condition:
	\begin{equation}\label{transverse}
		|u(x)|\leq\frac{C_1}{|x'|},\quad x\in\mathbb{R}^n\setminus\{x'=0\}.
	\end{equation}
	This is weaker than the $\mathcal L^\infty_r$ bound: it permits
	$|x'|^{-1}$ growth near every point of the axis and imposes no bound
	in the axial direction. This singular geometry is substantially different from the
isolated-point setting discussed above.  Indeed, the transverse bound
alone implies only
\[
    u\in L^q_{\mathrm{loc}}
    \qquad\text{for every }q<n-1,
\]
near the axis. Moreover, the potential singular set is one-dimensional,
rather than an isolated point.

	Our main result is the following.

	\begin{theorem}\label{thm:main}
		Let $n\geq 5$ and let $(u,p)$ be a smooth solution to~\eqref{NS} in
		$\mathbb{R}^n\setminus\{x'=0\}$ satisfying~\eqref{transverse}.  Then $u\equiv 0$.
	\end{theorem}

	\begin{remark}\label{rem:Landau4D}
The dimension restriction in Theorem~\ref{thm:main} is sharp.  Let
$(u_3,p_3)$ be a Landau solution~\cite{Landau} in
$\R^3\setminus\{0\}$.  Since $|x|^{-1}\le |x'|^{-1}$, the Landau
solution itself gives a three-dimensional counterexample after restriction
to the complement of the axis.  Extending it trivially in $x_4$-direction gives a four-dimensional solution with a singular line and with
$|u(x)|\lesssim |x'|^{-1}$.  More explicitly, set
\[
 u(x_1,x_2,x_3,x_4)=(u_3(x_1,x_2,x_3),0),\quad
 p(x_1,x_2,x_3,x_4)=p_3(x_1,x_2,x_3).
\]
Then $(u,p)$ solves~\eqref{NS} in
$\R^4\setminus\{x'=0\}$, where $x'=(x_1,x_2,x_3)$, and
$|u(x)|\lesssim |x'|^{-1}$.  Thus the transverse rigidity statement fails
in dimensions three and four.
\end{remark}

	The  key to prove Theorem~\ref{thm:main} is the following local quantitative regularity.
	We write $B_r=\{x\in\R^n:|x|<r\}$.

	\begin{theorem}[Quantitative local regularity]\label{thm:local-regularity}
		Let $n\geq5$ and let $(u,p)$ be a smooth solution of~\eqref{NS} in
		$B_1\setminus\{x'=0\}$.  If
		\[
		|u(x)|\leq \frac{M}{|x'|},
		\quad x\in B_1\setminus\{x'=0\},
		\]
		then $(u,p)$ extends smoothly across the axis to a solution in
		$B_{1/128}$.  Moreover,
		\begin{equation}\label{eq:local-quantitative-bound}
		\|u\|_{L^\infty(B_{1/128})}
		+\|\nabla u\|_{L^\infty(B_{1/128})}
		\leq C(n,M).
		\end{equation}
	\end{theorem}

\begin{remark}
A related removable-singularity problem for rays was studied by
Li, Li, and Yan~\cite{LiLiYan25}, who established a sharp criterion for local
$(-1)$-homogeneous solutions near a singular ray.
\end{remark}
\begin{remark}[Working-scale form]\label{rem:B8-working-scale}
For the proof it is convenient to work on the larger fixed ball $B_8$.
More precisely, we shall first prove that if $(u,p)$ is a smooth solution
of~\eqref{NS} in $B_8\setminus\{x'=0\}$ satisfying
\[
 |u(x)|\le \frac{M}{|x'|},
 \qquad x\in B_8\setminus\{x'=0\},
\]
then $(u,p)$ extends smoothly to $B_{1/16}$ and
\begin{equation}\label{eq:B8-working-scale}
\|u\|_{L^\infty(B_{1/16})}
+\|\nabla u\|_{L^\infty(B_{1/16})}
\le C(n,M).
\end{equation}
Theorem~\ref{thm:local-regularity} follows immediately from this
working-scale estimate by the fixed Navier--Stokes scaling.
\end{remark}

\subsection{Applications of main results}
The quantitative local theorem is the main technical  result of the paper.
It has two further consequences, proved in Section~\ref{sec:applications}.
In both applications, the constants in the local estimate are uniform
under translations along the axis and the scaling of the Navier-Stokes equations.

First, it gives a removable-singularity theorem along a portion of a line.
If $\Sigma$ is contained in the $x_n$-axis, $(u,p)$ is smooth away from
$\Sigma$, and locally
\[
 |u(x)|\leq \frac{C}{\operatorname{dist}(x,\Sigma)},
\]
then, for $n\geq5$, the pair $(u,p)$ extends smoothly across $\Sigma$.

Second, the same argument applies in exterior domains. If
$\Omega\subset\mathbb R^n$ is an exterior domain and a solution smooth
in $\Omega\setminus\{x'=0\}$ satisfies $|u(x)|\leq C_1|x'|^{-1}$ for sufficiently large $|x|$, the
uniform local estimates upgrade this to the isotropic bound
\[
 |u(x)|\leq \frac{C}{|x|}
 \quad\text{for all sufficiently large }|x|.
\]
The exterior asymptotic expansions in 
Bang-Gui-Liu-Wang-Xie~\cite{BangGuiLiuWangXie} then yield
\[
 u_i(x)=b_jG_{ij}(x)+O(|x|^{1-n}),
 \quad |x|\to\infty,
\]
where $G_{ij}$ is the Stokes Green tensor and $b$ is the conserved flux of
the Navier-Stokes stress tensor through a sufficiently large sphere.
See Theorem~\ref{thm:exterior-optimal-decay} for the detailed statement.

\subsection{Key ideas and outline of the proof}
The central difficulty is that the formal inequality
$-\Delta H+u\cdot\nabla H\leq0$ cannot initially be used across the
possibly singular axis.  The proof is organized around the following steps.

\emph{Step 1: scale-invariant local control.}  The transverse bound gives
pointwise estimates for derivatives of $u$ and for a suitably normalized
pressure.  A  cutoff argument then yields finite local energy, weak
extension of the Navier-Stokes equations across the axis, and the local
energy identity.

\emph{Step 2: total-head estimate.}  Following Frehse and
R\r{u}\v{z}i\v{c}ka~\cite{FrehseRuzicka98}, we construct approximate Green
functions for the adjoint drift operator
$-\Delta-u_k\cdot\nabla$.  Testing the equation of the  total-head pressure with a
localized Green function yields the  bound
\begin{equation}\label{Hbound-intro}
 H_+\in L^\infty(B_1),
\end{equation}
with norm controlled only by $n$ and the constant in the transverse bound.

\emph{Step 3: local regularity.}  A localized Frehse-R\r{u}\v{z}i\v{c}ka
estimate converts \eqref{Hbound-intro} into Morrey bound for $\nabla u$.
A finite Stokes-Morrey bootstrap then yields the working-scale estimate
\eqref{eq:B8-working-scale}. A fixed rescaling gives
\eqref{eq:local-quantitative-bound} and the smooth extension across the
axis, proving Theorem~\ref{thm:local-regularity}.

\emph{Step 4: global and geometric consequences.}  For a whole-space
solution the local theorem can be applied at arbitrarily large scales.  Using
the $B_8$-scale form proved below, for each fixed $x$ one obtains
$|u(x)|\leq C/R$ whenever $R>16|x|$; letting $R\to\infty$ gives $u(x)=0$.
Translating and rescaling the local theorem gives the removable singularity on the axis, while the same estimate near the two poles of large spheres yields
the exterior decay needed for the asymptotic expansion.

\subsection{Organization of the paper}
Section~\ref{sec:estimates} develops the local derivative and pressure
bounds, the weak extension across the axis, and the local energy identity.
The borderline dimension $n=5$ requires particular care in the transverse
Hardy estimate and in working with the equation of the total-head pressure across the axis.
Section~\ref{sec:Green} constructs the adjoint Green function and records the
estimates uniform in the regularization parameters.  Section~\ref{sec:proof-main}
proves the quantitative local theorem.  Section~\ref{sec:proof} then deduces whole-space rigidity
by scaling.  The removable-singularity and exterior results are collected in
Section~\ref{sec:applications}.  The localized weighted estimate used in the
regularity step is proved in Appendix~\ref{sec:appendix}.

\subsection{Notation}
	Throughout the paper, $C$ denotes a generic positive constant that may change from
	line to line. In the local argument it depends only on $n$ and $M$;
	for whole-space and exterior solutions it depends only on $n$ and $C_1$,
	unless explicitly stated otherwise. Constants in asymptotic remainders
	may also depend on the solution on a fixed bounded annulus.
	Repeated indices are summed from $1$ to $n$, and $f_+:=\max\{f,0\}$.
	$B_r(x)$ denotes the open ball of radius $r$ centered at $x$; when the center is
	omitted, $B_r$ means $B_r(0)$.
	We write $x=(x',x_n)\in\mathbb{R}^n$ with $x'\in\mathbb{R}^{n-1}$, so that
	$|x'|=\operatorname{dist}(x,\{x'=0\})$ is the transverse distance to the $x_n$-axis.

	\section{Preliminary local estimates}
\label{sec:estimates}

By Remark~\ref{rem:B8-working-scale}, it suffices to establish the
working-scale estimate \eqref{eq:B8-working-scale}. Accordingly, throughout
the local argument below we work with a solution in
$B_8\setminus\{x'=0\}$ satisfying
\[
 |u(x)|\le M|x'|^{-1}.
\]

    \subsection{Derivative and pressure estimates}
	We first record the  estimates of derivatives and  pressure  needed below. They follow from scale-invariant interior regularity for the
	stationary Navier-Stokes system; see, for example,
	\cite{Galdi,SverakTsai00} for the  interior Stokes estimates.

\begin{lemma}[Local normalization of the pressure]\label{lem:local-pressure}
Let $n\geq3$. Suppose that $(u,p)$ is a smooth solution of \eqref{NS} in $B_8\setminus\{x'=0\}$ and $|u(x)|\le M|x'|^{-1}$. Then there is $c\in\mathbb R$ such that
\begin{equation}\label{eq:local-pressure-bound}
 |p(x)-c|\le C(n,M)|x'|^{-2},\quad x\in B_5\setminus\{x'=0\}.
\end{equation}
Moreover, for every integer $k\ge0$,
\begin{equation}\label{eq:local-derivative-bound}
 |\nabla^k u(x)|\le C(k,n,M)|x'|^{-k-1},\quad
 |\nabla^{k+1}p(x)|\le C(k,n,M)|x'|^{-k-3}
\end{equation}
in $B_6\setminus\{x'=0\}$.
\end{lemma}
\begin{proof}
Fix $x_0\in B_6\setminus\{x'=0\}$ and put $r=|x^{\prime}_0|/8$. Then $B_{2r}(x_0)\Subset B_8\setminus\{x'=0\}$. For $v(y)=r u(x_0+ry)$ and $P(y)=r^2p(x_0+ry)$, the transverse bound gives $\|v\|_{L^\infty(B_2)}\le C M$. Interior Stokes estimates applied first to
$-\Delta v+\nabla P=-\operatorname{div}(v\otimes v)$ give local
$W^{1,q}$ bounds for $v$ and $L^q$ bounds for $P$ modulo constants for
 $1<q<\infty$. Since $v$ is bounded, another application to
$-\Delta v+\nabla P=-(v\cdot\nabla)v$, followed by differentiation
and  bootstrapping, yields
$|\nabla^k v(0)|+|\nabla^{k+1}P(0)|\le C(k,n,M)$.
Scaling back proves \eqref{eq:local-derivative-bound}; in particular,
\begin{equation}\label{eq:local-gradp}
 |\nabla p(x)|\le C(n,M)|x'|^{-3}.
\end{equation}
Let $e_1=(1,0,\ldots,0)\in\mathbb R^{n-1}$ and set $c=p(e_1,0)$. For every $x\in B_5\setminus\{x'=0\}$, let $\rho = |x'|$, $\theta = \frac{x'}{\rho}$. Then $x=(\rho\theta,x_n)$. One can connect $x$ with $e_1=(1,0,\ldots,0)$ through the following paths: first to $(\theta,x_n)$ radially, then to $(e_1,x_n)$ along the circle on $S^{n-2}$, and finally to $(e_1,0)$ along the $x_n$-direction. All three paths lie in $B_6\setminus\{x'=0\}$, along them, the distance to the origin is at most $\max\{|x|,\sqrt{1+x_n^2}\}<\sqrt{26}<6$. Since $n\ge3$, $S^{n-2}$ is connected and the spherical curve can be chosen with length at most $\pi$. By \eqref{eq:local-gradp}, the three pressure variations are bounded respectively by
\[
 C\left|\int_\rho^1s^{-3}\,ds\right|,\quad C\pi,\quad C|x_n|.
\]
Thus $|p(x)-c|\le C(\rho^{-2}+1)\le C\rho^{-2}$ for $0<\rho<5$, which proves \eqref{eq:local-pressure-bound}.
\end{proof}

In all subsequent arguments we subtract the constant  as in
Lemma~\ref{lem:local-pressure}.  The notation
$H=\frac12|u|^2+p$ always refers to this normalized pressure.

\subsection{Weak extension and the local energy identity}

\begin{lemma}[Weak extension across the axis]\label{lem:weak-extension}
Let $n\geq5$ and let $(u,p)$ satisfy the hypotheses of
Lemma~\ref{lem:local-pressure}.  After replacing $p$ by $p-c$, one has
\begin{equation}\label{eq:uniform-local-norms}
 \|u\|_{W^{1,2}(B_4)}+\|u\|_{L^4(B_4)}
 +\|p\|_{L^2(B_4)}+\|p\|_{W^{1,n/(n-1)}(B_4)}
 +\|H\|_{L^2(B_4)}\leq C(n,M).
\end{equation}
For every $1\leq q<n-1$ one also has
\begin{equation}\label{eq:uniform-local-Lq}
 \|u\|_{L^q(B_4)}\leq C(n,M,q).
\end{equation}
Moreover, $(u,p)$ is a distributional solution of~\eqref{NS} in $B_4$.
The local energy identity holds for every $\varphi\in C_c^\infty(B_4)$, namely
\begin{equation}\label{eq:local-energy-equality}
 \int_{B_4}\nabla u:\nabla(u\varphi)\,dx
 =\int_{B_4}H\,u\cdot\nabla\varphi\,dx.
\end{equation}
Consequently, the local Navier-Stokes inequality \eqref{eq:FR-local-NS-inequality} holds for every nonnegative $\varphi\in C_c^\infty(B_4)$.
\end{lemma}

\begin{proof}
Let $r=|x'|$. Choose
$\eta\in C_c^\infty(B_5)$ with $0\leq \eta \leq1$ and $\eta\equiv1$ on $B_{9/2}$, and let
$\chi_\varepsilon(r)$ satisfy
\[
0\leq\chi_\varepsilon\leq1,\quad
\chi_\varepsilon=0\ \text{on }[0,\varepsilon],\quad
\chi_\varepsilon=1\ \text{on }[2\varepsilon,\infty),\quad
|\nabla\chi_\varepsilon|\lesssim\varepsilon^{-1},\quad
|\Delta\chi_\varepsilon|\lesssim\varepsilon^{-2}.
\]
Set
\[
 \phi_\varepsilon:=\eta^2\chi_\varepsilon^2,\quad
 A_\varepsilon:=\{\varepsilon<r<2\varepsilon\}\cap B_5.
\]
We test the classical equation on $B_5\cap \{r>\varepsilon\}$ with
$u\phi_\varepsilon$.  Using integrations
by parts twice and  $\operatorname{div}u=0$, give
\begin{equation}\label{eq:weak-extension-energy-cutoff}
 \int_{B_5\cap \{r>\varepsilon\}} |\nabla u|^2\phi_\varepsilon\,dx
 =\frac12\int_{B_5\cap \{r>\varepsilon\}} |u|^2\Delta\phi_\varepsilon\,dx
   +\int_{B_5\cap \{r>\varepsilon\}} H\,u\cdot\nabla\phi_\varepsilon\,dx .
\end{equation}

We have
\[
 |\nabla\phi_\varepsilon|
 \leq C\bigl(|\nabla\eta|\chi_\varepsilon^2
              +|\nabla\chi_\varepsilon|\chi_\varepsilon\bigr),
\]
and
\[
 |\Delta\phi_\varepsilon|
 \leq C\bigl((1+|\nabla\eta|^2)\chi_\varepsilon^2
              +|\Delta\chi_\varepsilon|+|\nabla\chi_\varepsilon|^2\bigr).
\]
The pointwise estimates from Lemma~\ref{lem:local-pressure} give that
$|u|\lesssim r^{-1}$ and $|H|+|p|\lesssim r^{-2}$.  The terms containing
$\nabla\eta$ are integrable since, in the  polar coordinate for $\mathbb{R}^{n-1}$,
\[
 \int_0^1 r^{-2}r^{n-2}\,dr<\infty,
 \quad
 \int_0^1 r^{-3}r^{n-2}\,dr<\infty
 \quad (n\geq5).
\]
On $A_\varepsilon$, the  volume estimate
$|A_\varepsilon|\lesssim\varepsilon^{n-1}$ gives
\begin{equation}\label{eq:weak-extension-shell}
\begin{aligned}
&\int_{A_\varepsilon}|u|^2
 \bigl(|\Delta\chi_\varepsilon|+|\nabla\chi_\varepsilon|^2\bigr)\,dx
 +\int_{A_\varepsilon}|H|\,|u|\,|\nabla\chi_\varepsilon|\,dx
 \leq C(n,M)\varepsilon^{n-5}.
\end{aligned}
\end{equation}
Consequently, \eqref{eq:weak-extension-energy-cutoff} implies
\[
 \sup_{0<\varepsilon<1}
 \int_{B_{9/2}\cap\{r>2\varepsilon\}}|\nabla u|^2\,dx
 \leq C(n,M).
\]
The monotone
convergence theorem therefore gives
$\nabla u\in L^2(B_{9/2})$.  The pointwise bound gives
$u\in L^2(B_{9/2})$ as well.  To identify these classical derivatives
with the distributional derivatives across the axis, integrate by parts
against $\psi\chi_\varepsilon$, where $\psi\in C_c^\infty(B_{9/2})$.
The additional term is bounded by
\[
 \|\psi\|_{L^\infty}\int_{A_\varepsilon}
 |u|\,|\nabla\chi_\varepsilon|\,dx
 \leq C\|\psi\|_{L^\infty}\varepsilon^{n-3}\longrightarrow0.
\]
Thus $u\in W^{1,2}(B_{9/2})$, with norm bounded by $C(n,M)$.

For $n\geq6$, $|u|\leq M/r$ directly gives
$u\in L^4(B_{9/2})$. For $n=5$, let $Eu$ be a compactly supported
Sobolev extension to $\mathbb R^n$ satisfying
$\|Eu\|_{W^{1,2}(\mathbb R^n)}\leq C\|u\|_{W^{1,2}(B_{9/2})}$.
Applying the Hardy inequality in the $n-1$ transverse variables and
integrating in $x_n$ gives
\[
 \begin{aligned}
 \int_{B_{9/2}}\frac{|u|^2}{r^2}\,dx
 &\leq \int_{\mathbb R^n}\frac{|Eu|^2}{r^2}\,dx
 \leq \frac{4}{(n-3)^2}
        \int_{\mathbb R^n}|\nabla_{x'}Eu|^2\,dx \leq C\|u\|_{W^{1,2}(B_{9/2})}^2.
 \end{aligned}
\]
Hence
\[
 \int_{B_{9/2}}|u|^4\,dx
 \leq M^2\int_{B_{9/2}}\frac{|u|^2}{r^2}\,dx
 \leq C(n,M),
\]
and in every dimension $n\geq5$ we have $u\in L^4(B_{9/2})$.

We next extend the equations across the axis.  Let
$\psi\in C_c^\infty(B_{9/2})$ and test \eqref{NS} with
$\psi\chi_\varepsilon$.  After integration by parts, the terms that are
new compared with the desired weak formulation are supported in
$A_\varepsilon$ and satisfy
\[
 \int_{A_\varepsilon}
 \bigl(|\nabla u|+|u|^2+|p|\bigr)|\nabla\chi_\varepsilon|\,dx
 \leq C(n,M)\varepsilon^{n-4}\longrightarrow0,
\]
whereas the error in the divergence equation is
$O(\varepsilon^{n-3})$. Thus
$(u,p)$ satisfies the distributional Navier-Stokes equations in
$B_{9/2}$, and in particular in $B_4$.

Taking the divergence of the extended momentum equation gives
\[
 -\Delta p=\partial_i\partial_j(u_i u_j)
 \quad\text{in }\mathcal D'(B_{9/2}).
\]
Since $u\in L^4(B_{9/2})$, the tensor $u\otimes u$ belongs to
$L^2(B_{9/2})$.  A standard local pressure decomposition on $B_{9/2}$
gives, after restriction to $B_{17/4}$,
\[
 p=p_0+h,
\]
where $p_0\in L^2(B_{17/4})$, $h$ is harmonic there, and
\[
 \|p_0\|_{L^2(B_{17/4})}
 \le C\|u\otimes u\|_{L^2(B_{9/2})}
 \le C(n,M).
\]
By the normalized pointwise pressure bound,
$\|p\|_{L^1(B_{9/2})}\le C(n,M)$. Hence
$\|h\|_{L^1(B_{17/4})}\le C(n,M)$, and the interior estimate for harmonic
functions yields
\[
 \|h\|_{L^2(B_4)}\le C\|h\|_{L^1(B_{17/4})}\le C(n,M).
\]
Consequently
\[
 \|p\|_{L^2(B_4)}\le C(n,M).
\]
Therefore
$H=\frac12|u|^2+p\in L^2(B_4)$.

Moreover, the pointwise estimate
\[
|\nabla p(x)|\leq C(n,M)|x'|^{-3}
\]
implies that, for every $s<(n-1)/3$,
\[
\int_{B_4\setminus\{x'=0\}}|\nabla p|^s\,dx
\lesssim
\int_0^1 r^{n-2-3s}\,dr+1
\leq C(n,M,s).
\]
In particular, since $n\geq5$,
\[
\frac n4<\frac{n-1}{3},
\]
and hence
\begin{equation}\label{eq:grad-p-Ln4}
\|\nabla p\|_{L^{n/4}(B_4\setminus\{x'=0\})}
\leq C(n,M).
\end{equation}

We next verify that the classical derivatives of $p$ away from the axis
are also its distributional derivatives in the whole ball $B_4$.
Let $\psi\in C_c^\infty(B_4)$, and let
$\chi_\varepsilon=\chi_\varepsilon(|x'|)$ be the transverse cutoff used
above.
For each $i=1,\ldots,n$, the function
$\psi\chi_\varepsilon$ is compactly supported away from the axis, and
therefore  integration by parts gives
\[
\int_{B_4}
p\,\partial_i(\psi\chi_\varepsilon)\,dx
=
-\int_{B_4}
(\partial_i p)\,\psi\chi_\varepsilon\,dx.
\]
Equivalently,
\begin{equation}\label{eq:pressure-weak-derivative-cutoff}
\int_{B_4}
p\,\chi_\varepsilon\,\partial_i\psi\,dx
=
-\int_{B_4}
(\partial_i p)\,\psi\chi_\varepsilon\,dx
-
\int_{B_4}
p\,\psi\,\partial_i\chi_\varepsilon\,dx.
\end{equation}

The last term vanishes as $\varepsilon\downarrow0$.  Indeed, for
$i=n$ it is identically zero, since $\chi_\varepsilon$ depends only on
$x'$.  For $1\leq i\leq n-1$, the normalized pressure estimate
$|p(x)|\leq C(n,M)|x'|^{-2}$ gives
\[
\begin{aligned}
\left|
\int_{B_4}
p\,\psi\,\partial_i\chi_\varepsilon\,dx
\right|
\leq
C\|\psi\|_{L^\infty}
\varepsilon^{-1}
\int_{A_\varepsilon}|x'|^{-2}\,dx
\leq
C(n,M,\psi)\,
\varepsilon^{-1}\varepsilon^{-2}
|A_\varepsilon|
\leq
C(n,M,\psi)\varepsilon^{n-4}
\longrightarrow0.
\end{aligned}
\]

Since $p\in L^2(B_4)$ and
$\partial_i p\in L^{n/4}(B_4\setminus\{x'=0\})$, while
$\chi_\varepsilon\to1$ almost everywhere, we may let
$\varepsilon\downarrow0$ in
\eqref{eq:pressure-weak-derivative-cutoff} to obtain
\[
\int_{B_4}p\,\partial_i\psi\,dx
=
-\int_{B_4}(\partial_i p)\,\psi\,dx.
\]
Thus the classical gradient of $p$ away from the axis extends to the
distributional gradient of $p$ in $B_4$. As
\(
\frac n4\geq\frac{n}{n-1}, (n\geq5),
\)
\eqref{eq:grad-p-Ln4} therefore implies
\[
\|\nabla p\|_{L^{n/(n-1)}(B_4)}
\leq C(n,M).
\]
Together with the already established $L^2$ bound for $p$, this yields
\begin{equation}\label{eq:pressure-W1q}
\|p\|_{W^{1,n/(n-1)}(B_4)}
\leq C(n,M).
\end{equation}

Finally, $u\otimes u,p\in L^2(B_4)$, so the test functions in the  weak form of  momentum
equation extends continuously to compactly supported $W^{1,2}$ 
fields. Thus $u\varphi$ is admissible. Moreover,
$|u|^2\varphi\in W^{1,4/3}_0(B_4)$ and $u\in L^4(B_4)$, so the
divergence condition in weak form gives
$\int u\cdot\nabla(|u|^2\varphi)=0$ by density.
Combining these two identities gives, for every
$\varphi\in C_c^\infty(B_4)$,
\[
 \int_{B_4}\nabla u:\nabla(u\varphi)\,dx
 =\int_{B_4}H\,u\cdot\nabla\varphi\,dx.
\]
This is the asserted local energy identity, and it implies the local
Navier-Stokes inequality for nonnegative $\varphi$.
\end{proof}

\section{Green function for the adjoint drift operator}\label{sec:Green}

	For Sections~\ref{sec:Green}-\ref{sec:proof-main}, fix
	a pair $(u,p)$ satisfying the hypotheses of
	Lemma~\ref{lem:local-pressure} with $n\geq5$,  and normalize $p$ by that lemma.
	We now construct the approximate Green functions used to obtain pointwise
	control of the total head pressure.
	We follow the methods of Frehse and R\r{u}\v{z}i\v{c}ka~\cite{FrehseRuzicka96,FrehseRuzicka98}.

	For each $x_0\in B_{1}\setminus\{x'=0\}$, consider the Dirichlet problem
	\begin{equation}\label{eq:Green2}
		\begin{cases}
			-\Delta G-u_k\cdot\nabla G=\delta_{h,x_0}(x),&\text{in }B_4,\\[4pt]
			G=0,&\text{on }\p B_4,
		\end{cases}
	\end{equation}
where $u_k$ and $\delta_{h,x_0}$ are defined as follows:
	\begin{itemize}
		\item By Lemma~\ref{lem:weak-extension}, $u$ is divergence free in
		$B_{9/2}$ in the sense of distributions.  Standard mollification in
		$B_{9/2}$ therefore gives a sequence
		$\{u_k\}\subset C^\infty(\overline{B_4})$ of divergence-free vector
		fields such that
		\[
		u_k\to u\quad\text{in }W^{1,2}(B_4)\cap L^4(B_4)
		\cap L^q(B_4)\quad\text{for every }1\le q<n-1.
		\]

		\item Fix $\rho\in C_c^\infty(B_1)$ with $\rho\ge0$ and
		$\int_{B_1}\rho=1$, and set
		\[
		\delta_{h,x_0}(x)
		=h^{-n}\rho\!\left(\frac{x-x_0}{h}\right),
		\quad 0<h<h_0<\frac12.
		\]
		Thus $\delta_{h,x_0}\ge0$, $\int\delta_{h,x_0}=1$,
		$\operatorname{supp}\delta_{h,x_0}\subset B_h(x_0)$, and
		$\|\delta_{h,x_0}\|_{L^\infty}\le C h^{-n}$.
	\end{itemize}

	For fixed $h,k,x_0$, the Dirichlet problem~\eqref{eq:Green2} has a unique weak solution in $W_0^{1,2}(B_4)$. Standard elliptic regularity then gives smoothness.  We denote the solution by $G_{h,k,x_0}$ (or simply $G$).

	\begin{proposition}\label{Greenfunction}
    The solution $G_{h,k,x_0}$ to \eqref{eq:Green2} satisfies the following three properties.

	\begin{enumerate}
		\item[(P1)]\textbf{Regularity and sign.}
		$G_{h,k,x_0}\in C^\infty(\overline{B_4})$ and $G_{h,k,x_0}\geq 0$.

		\item[(P2)]
        \textbf{Bounds for a fixed source scale.}
		For every fixed $h\in(0,h_0)$, there is a constant $C(h)$,
		independent of $k$ and $x_0$, such that
		\begin{equation*}
		\|G_{h,k,x_0}\|_{L^\infty(B_4)}
		+\|G_{h,k,x_0}\|_{W^{1,4}(B_4)}\leq C(h).
		\end{equation*}

		\item[(P3)]\textbf{Uniform annular bounds.}
		For every $1\le q<\frac{n}{n-1}$, the family $\{G_{h,k,x_0}\}$ is uniformly bounded in
		$W^{1,q}(B_4 )$, and it is uniformly bounded in
		\[
		L^4(B_{3}\setminus B_{2})\cap
		W^{1,2}(B_{3}\setminus B_{2}).
		\]
		These bounds are uniform with respect to all parameters $h$, $k$, and $x_0$.
	 \end{enumerate}
	\end{proposition}

Unless indicated
otherwise, the constants in this section depend only on $n$ and $M$ (and on the fixed
exponent in (P3)).
	(P1) follows from the maximum principle for the operator
	$-\Delta-u_k\cdot\nabla$ (note that $u_k$ is smooth and
	$\delta_{h,x_0}\geq 0$).

	Next, we prove Properties~(P2) and~(P3).
	Fix $h>0$ and consider the equation
	\begin{equation}\label{eq:Green3}
		-\Delta G-u_k\cdot\nabla G=\delta_{h,x_0}\quad\text{in }B_4,\quad
		G=0\ \text{on }\partial B_4.
	\end{equation}

\begin{proof}[Proof of (P2) of Proposition \ref{Greenfunction}.]
	We first obtain an $L^{2n/(n-2)}$ estimate, then perform a Moser
	iteration to obtain the $L^\infty$ bound, and finally use
	Calder\'on-Zygmund theory to obtain the $W^{1,4}$ estimate.

	{\it Step 1: Energy estimate.}
	Test~\eqref{eq:Green3} with $G$ itself (recall $G\geq 0$ by (P1)).
	Since $u_k$ is divergence-free,
	\[
	\int_{B_4 } u_k\!\cdot\!\nabla G\; G\,dx
	=\frac12\int_{B_4 } u_k\!\cdot\!\nabla(G^2)\,dx
	=-\frac12\int_{B_4 }(\divg u_k)G^2\,dx=0 .
	\]
	Hence
	\begin{equation}\label{eq:seed}
		\int_{B_4 }|\nabla G|^2\,dx
		=\int_{B_4 }\delta_{h,x_0}\,G\,dx .
	\end{equation}
	  By H\"older's inequality and the Sobolev embedding
	$W^{1,2}_0(B_4 )\hookrightarrow L^{\frac{2n}{n-2}}(B_4)$,
	\[
	\int_{B_4 }\delta_{h,x_0}G\,dx
	\leq \|\delta_{h,x_0}\|_{L^{\frac{2n}{n+2}}(B_4 )}\,
	\|G\|_{L^{\frac{2n}{n-2}}(B_4 )}
	\leq C(h)\|\nabla G\|_{L^2(B_4 )},
	\]
	where $C(h)$ depends on $h$
	but not on $k$ or $x_0$.  Substituting into~\eqref{eq:seed} yields
	$\|\nabla G\|_{L^2(B_4 )}\leq C(h)$, and applying the Sobolev
	inequality once more gives
	\begin{equation}\label{eq:seedLq}
		\|G\|_{L^{\frac{2n}{n-2}}(B_4 )}\leq C(h),\ \ \ \text{uniformly in}\ k\, \text{and}\,  x_0.
	\end{equation}

{\it Step 2: Moser iteration and the $L^\infty$ bound.}
	For $s\ge1$, test \eqref{eq:Green3} with $G^s$. Integration by parts and using the divergence free property of $u_k$
	give
	\begin{equation}\label{eq:Lqs}
		\frac{4s}{(s+1)^2}\int_{B_4 }|\nabla G^{\frac{s+1}{2}}|^2\,dx
		=\int_{B_4 }\delta_{h,x_0}G^s\,dx .
	\end{equation}
	Since $G\geq 0$ and $\delta_{h,x_0}\leq C_0 h^{-n}$ pointwise,
	\[
	\int_{B_4 }\delta_{h,x_0}G^s\,dx
	\leq C_0 h^{-n}\int_{B_h(x_0)}G^s\,dx
	\leq C_0 h^{-n}\int_{B_4 }G^s\,dx
	= C_0 h^{-n}\,\|G\|_{L^s(B_4 )}^s .
	\]
	Applying the Sobolev inequality to $G^{\frac{s+1}{2}}\in W^{1,2}_0(B_4 )$
	and using~\eqref{eq:Lqs}, yield
	\begin{equation}\label{eq:moserstep}
		\|G\|_{L^{\frac{(s+1)n}{n-2}}(B_4 )}^{s+1}
		\leq C\int_{B_4 }|\nabla G^{\frac{s+1}{2}}|^2\,dx
		\leq C\,\frac{(s+1)^2}{s}\,h^{-n}\,
		\|G\|_{L^s(B_4 )}^s ,
	\end{equation}
	where $C$ denotes the Sobolev constant  and is independent of
	$s,h,k,x_0$.

	Now define the iteration sequence
	\[
	s_0=\frac{2n}{n-2},\quad
	s_{j+1}= \frac{(s_j+1)n}{n-2}\quad(j\geq 0).
	\]
	For $n\geq 3$ one has $\frac{n}{n-2}>1$, hence $s_j\nearrow\infty$ as
	$j\to\infty$.   Choose $s=s_j$ in \eqref{eq:moserstep} to  get
	\[
	\|G\|_{L^{s_{j+1}}(B_4 )}
	\leq \bigl(C\,h^{-n}\,(s_j+1)\bigr)^{\frac{1}{s_j+1}}
	\,\|G\|_{L^{s_j}(B_4 )}^{\frac{s_j}{s_j+1}} .
	\]
	Put $A_j=\max\{1,\|G\|_{L^{s_j}(B_4)}\}$.  The preceding
	inequality gives
	\begin{equation}\label{eq:moser-recursion}
	 A_{j+1}\leq
	 \bigl(C h^{-n}(s_j+1)\bigr)^{\frac1{s_j+1}}
	 A_j^{\frac{s_j}{s_j+1}} .
	\end{equation}
	Since $s_{j+1}\geq \frac{n}{n-2}s_j$ and $n\geq5$,
	\[
	 \sum_{j=0}^{\infty}\frac{1+\log(s_j+1)}{s_j+1}<\infty.
	\]
	Iterating \eqref{eq:moser-recursion} therefore gives
	\[
	 \sup_{j\geq0}A_j\leq C(h),
	\]
	where $C(h)$ is independent of $k$ and $x_0$.  Since
	$s_j\to\infty$, this implies
	\begin{equation}\label{eq:Green-Linfty}
	 \|G_{h,k,x_0}\|_{L^\infty(B_4)}\leq C(h).
	\end{equation}

	{\it Step 3: $W^{1,4}$ estimate.}
	Rewrite~\eqref{eq:Green3} as
	\[
	-\Delta G = \divg(u_k G) + \delta_{h,x_0}.
	\]
	By Lemma~\ref{lem:weak-extension}, $\{u_k\}$ is uniformly bounded in
$L^4(B_4)$ for every $n\geq5$.  Combining this with
\eqref{eq:Green-Linfty} gives
\[
\|u_kG\|_{L^4(B_4)}
\leq \|u_k\|_{L^4(B_4)}\|G\|_{L^\infty(B_4)}
\leq C(h).
\]
Moreover,
	$\|\delta_{h,x_0}\|_{L^\infty(B_4 )}\leq C h^{-n}$.
	The standard $W^{1,4}$ estimate for the Poisson equation with
	right-hand side in divergence form yields
	\[
	\|G\|_{W^{1,4}(B_4)}
	\leq C\bigl(
	\|u_kG\|_{L^4(B_4)}
	+\|\delta_{h,x_0}\|_{L^4(B_4)}
	+\|G\|_{L^4(B_4)}
	\bigr)
	\leq C(h).
	\]
	This completes the proof of (P2).
    \end{proof}

\begin{proof}[Proof of (P3) of Proposition \ref{Greenfunction}.]
	We must show that, for every $q<\frac{n}{n-1}$, the family $\{G\}$ is uniformly bounded in
	$W^{1,q}(B_4 )$, and that it is uniformly bounded in
	$L^4(B_{3}\setminus B_{2})\cap W^{1,2}(B_{3}\setminus B_{2})$
	with respect to $h$, $k$, and $x_0$.
	The essential point is that these bounds are \emph{independent of $h$},
	unlike the bounds in (P2).

{\it Step~1: A weighted gradient estimate independent of $h$.}
	For a small parameter $\alpha\in(0,1)$ to be fixed later, define
	\[
	\varphi(t)=\frac{t}{(1+t^\alpha)^{1/\alpha}},\quad t\geq 0.
	\]
	One checks that $\varphi(0)=0$, $0\leq\varphi(t)\leq 1$ for all $t\geq0$,
	and $\varphi'(t)=1/(1+t^\alpha)^{1+1/\alpha}$.
	Since $G\geq 0$ (P1), we may test~\eqref{eq:Green3} with
	$\varphi(G)\in W^{1,2}_0(B_4 )$.
	Because $\divg u_k=0$, the drift term vanishes:
	\[
	\int_{B_4 } u_k\!\cdot\!\nabla G\;\varphi(G)\,dx
	=\int_{B_4 } u_k\!\cdot\!\nabla\bigl(\Phi(G)\bigr)\,dx=0,
	\]
	where $\Phi'=\varphi$.  The remaining terms give
	\begin{equation}\label{eq:weighted}
		\int_{B_4 }\frac{|\nabla G|^2}{(1+G^\alpha)^{1+1/\alpha}}\,dx
		=\int_{B_4}\delta_{h,x_0}\,\frac{G}{(1+G^\alpha)^{1/\alpha}}\,dx
		\leq\int_{B_4}\delta_{h,x_0}\,dx=1.
	\end{equation}
	The bound~\eqref{eq:weighted} is independent of $h,k,x_0$.

	{\it Step~2: $L^\beta$ and $W^{1,q}$ bounds ($\beta<\frac{n}{n-2}$,
		$q<\frac{n}{n-1}$).}
	Set
	\begin{equation}\label{eq:beta}
		\beta=\frac{n-n\alpha}{n-2},\quad
		r=\frac{n\beta}{n+\beta}.
	\end{equation}
	Note that $\beta\nearrow\frac{n}{n-2}$ and $r \nearrow\frac{n}{n-1}$
	as $\alpha\to0+$.  By the Gagliardo-Nirenberg-Sobolev inequality
	on $B_4 $ (where $G$ vanishes on the boundary),
	\[
	\|G\|_{L^\beta(B_4)}
	\leq C\,\|\nabla G\|_{L^r(B_4)},
	\]
	with $C$ depending on $n$ and $r$, but not on $h,k,x_0$.
	To estimate the right-hand side,
	set $W=(1+G^\alpha)^{1+1/\alpha}$ and write
	\[
	 |\nabla G|^r=(|\nabla G|^2/W)^{r/2}W^{r/2}.
	\]
	H\"older's inequality with exponents
	$2/r$ and $2/(2-r)$ and \eqref{eq:weighted} yield
	\[
	\int_{B_4}|\nabla G|^r\,dx
	\leq\Bigl(\int_{B_4}\frac{|\nabla G|^2}{W}\,dx\Bigr)^{\!r/2}
	\Bigl(\int_{B_4}W^{\frac{r}{2-r}}\,dx\Bigr)^{\!\frac{2-r}{2}}
	\leq\Bigl(\int_{B_4}W^{\frac{r}{2-r}}\,dx\Bigr)^{\!\frac{2-r}{2}}.
	\]
	  Now
	\begin{equation}\label{W}
	W^{\frac{r}{2-r}}=(1+G^\alpha)^{\frac{(1+1/\alpha)r}{2-r}}
	\leq C(\alpha, r)\bigl(1+G^{\frac{(\alpha+1)r}{2-r}}\bigr).
	\end{equation}
	A direct computation using~\eqref{eq:beta} shows that
	\[
	\frac{(\alpha+1)r}{2-r}=\beta .
	\]
	Therefore
	\[
	\int_{B_4 }W^{\frac{r}{2-r}}\,dx
	\leq C(\beta)\Bigl(|B_4|+\|G\|_{L^\beta(B_4)}^\beta\Bigr).
	\]
	Combining the above estimates, one has
	\[
	\|G\|_{L^\beta(B_4)}
	\leq C(\beta) \Bigl(1+\|G\|_{L^\beta(B_4 )}^\beta\Bigr)^{\!\frac{2-r}{2r}} .
	\]
	Using~\eqref{eq:beta} again,
	$\frac{2-r}{2r}=\frac{1+\alpha}{2\beta}$, so the exponent on
	$\|G\|_{L^\beta(B_4)}$ is $\frac{1+\alpha}{2}<1$ for $\alpha<1$.
	 Young's inequality now gives
	$\|G\|_{L^\beta(B_4)}\leq C$ uniformly in $h,k,x_0$.
	Substituting this bound into~\eqref{W} and the preceding H\"older estimate
	yields the corresponding uniform $L^r$ bound for $\nabla G$.
	Thus, for any prescribed $1\le q<\frac{n}{n-1}$, we may choose
	$\alpha>0$ sufficiently small so that $q<r<\frac{n}{n-1}$ (and hence
	$1<r<2$).  The uniform $L^r$ bounds for $G$ and $\nabla G$, together
	with the finite measure of $B_4$, then imply the required uniform
	$W^{1,q}(B_4)$ bound.

	{\it Step~3: Annular $L^4$ bound via bootstrap.}
Let $\tau_\rho$ be a smooth cut-off function with
\[
\tau_\rho=1\ \text{on }B_{3+\rho}\setminus B_{2-\rho},\quad
\tau_\rho=0\ \text{outside }B_{3+2\rho}\setminus B_{2-2\rho},
\]
where $0<\rho<(1-h_0)/2$ is fixed.  Set
$D_\rho=B_{3+\rho}\setminus B_{2-\rho}$ and
$D_{2\rho}=B_{3+2\rho}\setminus B_{2-2\rho}$.

For $\lambda>0$, test~\eqref{eq:Green3} with $G^\lambda\tau_\rho^2$. Since $G>0$ in the interior by the strong maximum principle, this is justified first for fixed $h,k$; approximation by smooth powers then gives the stated uniform estimates.
On $D_{2\rho}$ the source term vanishes: indeed,
$\supp\delta_{h,x_0}\subset B_{1+h_0}$, while
$D_{2\rho}\subset\R^n\setminus B_{2-2\rho}$ and
$2-2\rho>1+h_0$.
Integration by parts gives
\begin{equation}\label{eq:tau2-corrected}
    \begin{aligned}
        \frac{4\lambda}{(\lambda+1)^2}\int_{B_4}|\nabla G^{\frac{\lambda+1}{2}}|^2\tau_\rho^2\,dx
        &=\frac{2}{\lambda+1}\int_{B_4} G^{\lambda+1}\bigl(|\nabla\tau_\rho|^2
        +\tau_\rho\Delta\tau_\rho\bigr)\,dx\\
        &\quad-\frac{2}{\lambda+1}\int_{B_4} G^{\lambda+1}\tau_\rho\,u_k\!\cdot\!\nabla\tau_\rho\,dx.
    \end{aligned}
\end{equation}

Since $\nabla\tau_\rho$ and $\Delta\tau_\rho$ are supported in
$D_{2\rho}\setminus D_\rho$ and are bounded by $C(\rho)$, the right-hand side
of~\eqref{eq:tau2-corrected} is controlled by
\[
C(\rho)\Bigl(\|G\|_{L^{\lambda+1}(D_{2\rho})}^{\lambda+1}
+\|u_k\|_{L^q(B_4)}\|G\|_{L^{\frac{(\lambda+1)q}{q-1}}(D_{2\rho})}^{\lambda+1}\Bigr).
\]

Now restrict $G^{\frac{\lambda+1}{2}}\tau_\rho$ to $D_\rho$ (where $\tau_\rho=1$)
and apply the Sobolev inequality.  Since $G^{\frac{\lambda+1}{2}}\tau_\rho=0$ on
$\partial B_4$, we have
\[
\|G\|_{L^{\frac{(\lambda+1)n}{n-2}}(D_\rho)}^{\lambda+1}
\leq \|G^{\frac{\lambda+1}{2}}\tau_\rho\|_{L^{\frac{2n}{n-2}}(B_4)}^2
\leq C\int_{B_4}|\nabla(G^{\frac{\lambda+1}{2}}\tau_\rho)|^2\,dx .
\]
Expanding the gradient, we have
\[
\int_{B_4}|\nabla(G^{\frac{\lambda+1}{2}}\tau_\rho)|^2\,dx
\leq 2\int_{B_4}|\nabla G^{\frac{\lambda+1}{2}}|^2\tau_\rho^2\,dx
+C(\rho)\int_{D_{2\rho}}G^{\lambda+1}\,dx .
\]
Combining this with~\eqref{eq:tau2-corrected} and the bound on the right-hand side,
we obtain, for every $\lambda>0$,
\begin{equation}\label{eq:bootstrap-corrected}
    \|G\|_{L^{\frac{(\lambda+1)n}{n-2}}(D_\rho)}^{\lambda+1}
    \leq C(n,\lambda,\rho)\Bigl(
    \|G\|_{L^{\lambda+1}(D_{2\rho})}^{\lambda+1}
    +\|u_k\|_{L^q}\,\|G\|_{L^{\frac{(\lambda+1)q}{q-1}}(D_{2\rho})}^{\lambda+1}
    \Bigr).
\end{equation}
Here the constant may depend on $\lambda$, but is independent of $h$, $k$,
and $x_0$.

By Lemma~\ref{lem:weak-extension} and $n\geq 5$, the sequence $\{u_k\}$ is uniformly
bounded in $L^q(B_4)$ for every $q<n-1$; in particular we may fix
$q\in(\frac n2,n-1)$.  For such $q$ we have $\frac{q}{q-1}<\frac{n}{n-2}$,
which means that the exponent $\frac{(\lambda+1)q}{q-1}$ in the second term
of~\eqref{eq:bootstrap-corrected} is strictly smaller than 
$\frac{(\lambda+1)n}{n-2}$.  Since
\(q/(q-1)<n/(n-2)\), we may choose \(\alpha>0\)
sufficiently small so that
\[
\beta=\frac{n(1-\alpha)}{n-2}>\frac{q}{q-1}.
\]
Then
\[
\lambda_0:=\frac{\beta(q-1)}q-1>0.
\]
Now start the iteration with $\lambda_0$. By its definition,
\[
\frac{(\lambda_0+1)q}{q-1}=\beta,
\]
and Step~2 gives a uniform $L^\beta(B_4)$ bound. Hence both terms on the
right-hand side of~\eqref{eq:bootstrap-corrected} are finite and uniformly bounded at the initial step. We fix $q$ and $\alpha$ as
functions of $n$ alone.
Set $\lambda_{j+1}$ by
\begin{equation}\label{iteration}
\lambda_{j+1}  = (\lambda_j +1 ) \cdot \frac{ n}{n-2 } \cdot \frac{q-1}{q} -1 .
\end{equation}
Equivalently, if $s_j:=\lambda_j+1$, then
\[
 s_{j+1}=\frac{n}{n-2}\frac{q-1}{q}s_j.
\]
Since $\frac{ n}{n-2 }\frac{q-1}{q}>1$, we have
$\lambda_j\nearrow\infty$.
Choose \(N\) and positive numbers
\(\rho_0>\rho_1>\cdots>\rho_N\) such that
\(2\rho_{j+1}\leq\rho_j\).
Then
\[
D_{2\rho_{j+1}}\subset D_{\rho_j}.
\]
Applying~\eqref{eq:bootstrap-corrected} successively, with the exponents defined through 
\eqref{iteration} and with the nested annuli above, yields an $N$ for which
\[
 P_N:=\frac{(\lambda_N+1)n}{n-2}\geq4.
\]
Consequently,
\[
\|G\|_{L^4(B_{3}\setminus B_{2})}\leq C,
\]
uniformly in $h,k,x_0$.

{\it Step~4: Annular $W^{1,2}$ estimate.}
As proved in Step 3, $\|G\|_{L^4(D_\rho)}$ is uniformly bounded for some $\rho>0$.
Fix \(\rho>0\) sufficiently small so that
\[
B_{3}\setminus B_{2}\Subset D_\rho
\]
and
\[
\operatorname{supp}\delta_{h,x_0}\cap D_\rho=\varnothing
\]
for all \(0<h<h_0\) and \(x_0\in B_1\).
 On $D_\rho$,  using $\divg u_k=0$, one has
\begin{equation}\label{eq:div-form}
    -\Delta G = u_k\!\cdot\!\nabla G = \divg(u_k G) \quad\text{in } D_\rho.
\end{equation}
Thus $G$ satisfies a Poisson equation with right-hand side in divergence form.

Since $B_3\setminus B_2 \Subset D_\rho$,
the interior $W^{1,2}$ estimate yields
\begin{equation}\label{eq:CZ}
    \|\nabla G\|_{L^2(B_3 \setminus B_2 )}
    \leq C\Bigl(\|G\|_{L^2(D_\rho)}+\|u_k G\|_{L^2(D_\rho)}\Bigr).
\end{equation}
Both terms on the right are uniformly bounded. Indeed, for the second term, H\"older's inequality gives
\[
\|u_k G\|_{L^2(D_\rho)}
\leq \|u_k\|_{L^4(D_\rho)}\|G\|_{L^4(D_\rho)}\leq C.
\]
Therefore $\|\nabla G\|_{L^2(B_3 \setminus B_2 )}\leq C$ uniformly in
$h,k,x_0$.

Together with the $L^4$ bound from Step~3, this establishes that
$\{G_{h,k,x_0}\}$ is uniformly bounded in
$L^4(B_3\setminus B_2)\cap W^{1,2}(B_3\setminus B_2)$
with respect to all parameters $h$, $k$, and $x_0$.
This completes the proof of (P3).
\end{proof}

    \section{Quantitative local regularity}\label{sec:proof-main}

In this section, we prove the working-scale estimate
\eqref{eq:B8-working-scale}, and thereby complete the proof of
Theorem~\ref{thm:local-regularity}.

  	\subsection{A uniform upper bound for the total head pressure}
\label{sec:Hbound}

  We  use the uniform estimates of Green function in Section  \ref{sec:Green} to obtain a one-sided bound for the
total head pressure.
By Lemma~\ref{lem:weak-extension}, in every dimension $n\geq5$,
\[
u\in W^{1,2}(B_4)\cap L^4(B_4),
\quad
H\in L^2(B_4).
\]
In particular,
\[
uH\in L^{4/3}(B_4).
\]
Since $(u,p)$ is smooth away from the axis, the classical head-pressure
identity
\[
-\Delta H+u\cdot\nabla H
=
-\frac12
\sum_{i,j=1}^n
(\partial_i u_j-\partial_j u_i)^2
\leq0
\]
holds in $B_4\setminus\{x'=0\}$.
We first show that this inequality extends across the axis in the
distributional sense.  This step is slightly delicate in the borderline
dimension $n=5$.

Let $\phi\in C_c^\infty(B_4)$ be nonnegative, and let
$\chi_\varepsilon$ be the transverse cutoff introduced in the proof of
Lemma~\ref{lem:weak-extension}.  It holds that 

\begin{equation}
\label{eq:head-cutoff-norms}
\begin{aligned}
\|\nabla\chi_\varepsilon\|_{L^2(B_4)}
\leq C\varepsilon^{(n-3)/2},\quad
\|\Delta\chi_\varepsilon\|_{L^2(B_4)}
\leq C\varepsilon^{(n-5)/2},\quad
\|\nabla\chi_\varepsilon\|_{L^4(B_4)}
\leq C\varepsilon^{(n-5)/4}.
\end{aligned}
\end{equation}
In particular, when $n=5$ the last two quantities are merely uniformly bounded, rather than converging to zero.

Indeed, $\phi\chi_\varepsilon$ is a nonnegative smooth test function
compactly supported in
$B_4\setminus\{x'=0\}$.  Hence
\[
\int_{B_4}
H\bigl(
-\Delta(\phi\chi_\varepsilon)
-u\cdot\nabla(\phi\chi_\varepsilon)
\bigr)\,dx
\leq0.
\]
Expanding the derivatives gives
\begin{equation}
\label{eq:head-cutoff-expanded}
\begin{aligned}
\int_{B_4}
H\chi_\varepsilon
\bigl(-\Delta\phi-u\cdot\nabla\phi\bigr)\,dx
\leq
2\int_{B_4}
H\,\nabla\phi\cdot\nabla\chi_\varepsilon\,dx
+
\int_{B_4}
H\phi\,\Delta\chi_\varepsilon\,dx
+
\int_{B_4}
Hu\cdot\nabla\chi_\varepsilon\,\phi\,dx .
\end{aligned}
\end{equation}
Set
\[
A_\varepsilon
:=
B_4\cap
\{\varepsilon<|x'|<2\varepsilon\}.
\]
We now let $\varepsilon\downarrow0$.  The first  term on the right-hand-side of \eqref{eq:head-cutoff-expanded} is estimated as
\[
\begin{aligned}
\left|
\int_{B_4}
H\,\nabla\phi\cdot\nabla\chi_\varepsilon\,dx
\right|\leq
C\|H\|_{L^2(A_\varepsilon)}
\|\nabla\chi_\varepsilon\|_{L^2(B_4)}
\leq
C\varepsilon^{(n-3)/2}
\|H\|_{L^2(A_\varepsilon)}
\longrightarrow0.
\end{aligned}
\]
For the second  term on the right-hand-side of \eqref{eq:head-cutoff-expanded}, using~\eqref{eq:head-cutoff-norms}, one has
\[
\begin{aligned}
\left|
\int_{B_4}
H\phi\,\Delta\chi_\varepsilon\,dx
\right|\leq
C\|H\|_{L^2(A_\varepsilon)}
\|\Delta\chi_\varepsilon\|_{L^2(B_4)}
\leq
C\varepsilon^{(n-5)/2}
\|H\|_{L^2(A_\varepsilon)}
\longrightarrow0.
\end{aligned}
\]
Here, in the borderline case $n=5$, the factor
$\varepsilon^{(n-5)/2}$ is equal to one, but
\(
\|H\|_{L^2(A_\varepsilon)}\longrightarrow0
\)
by the absolute continuity of the $L^2$ integral.
Similarly, since $uH\in L^{4/3}(B_4)$,
\[
\begin{aligned}
\left|
\int_{B_4}
Hu\cdot\nabla\chi_\varepsilon\,\phi\,dx
\right|
\leq
C
\|uH\|_{L^{4/3}(A_\varepsilon)}
\|\nabla\chi_\varepsilon\|_{L^4(B_4)}\leq
C\varepsilon^{(n-5)/4}
\|uH\|_{L^{4/3}(A_\varepsilon)}
\longrightarrow0.
\end{aligned}
\]
Since $\chi_\varepsilon\to1$ almost everywhere and
\(
H\bigl(-\Delta\phi-u\cdot\nabla\phi\bigr)\in L^1(B_4),
\)
dominated convergence in the left-hand side of
\eqref{eq:head-cutoff-expanded} therefore yields
\begin{equation}
\label{eq:head-weak}
\int_{B_4}
H(-\Delta\phi-u\cdot\nabla\phi)\,dx
\leq0
\end{equation}
for every nonnegative $\phi\in C_c^\infty(B_4)$.

We now use  the approximate Green function as a test function in \eqref{eq:head-weak}.
Fix
\(
x_0\in B_1\setminus\{x'=0\},
\)
and let
\(
G=G_{h,k,x_0}
\)
be the solution constructed in Section~\ref{sec:Green}.  Choose
$\zeta\in C_c^\infty(B_3)$ such that
\[
0\leq\zeta\leq1,
\quad
\zeta=1\quad\text{on }B_2,
\quad
\operatorname{supp}\nabla\zeta
\subset
A:=B_3\setminus\overline{B_2}.
\]
Since $x_0\in B_1$ and $h<h_0<1/2$,
$\operatorname{supp}\delta_{h,x_0}\subset B_{3/2}\subset B_2$; hence
$\zeta\delta_{h,x_0}=\delta_{h,x_0}$.  Moreover, $G\geq0$ and $G$ is smooth
for fixed $h,k$, so $G\zeta$ is an admissible nonnegative test function in
\eqref{eq:head-weak}.  Therefore
\begin{equation}
\label{eq:head-Gzeta-start}
\int_{B_4}
H\bigl(
-\Delta(G\zeta)-u\cdot\nabla(G\zeta)
\bigr)\,dx
\leq0.
\end{equation}
Using
\[
-\Delta G-u_k\cdot\nabla G
=
\delta_{h,x_0},
\]
we have
\[
-\Delta G-u\cdot\nabla G
=
\delta_{h,x_0}
+
(u_k-u)\cdot\nabla G.
\]
Hence
\begin{align*}
-\Delta(G\zeta)-u\cdot\nabla(G\zeta)
=
\zeta\delta_{h,x_0}
+\zeta(u_k-u)\cdot\nabla G
-2\nabla G\cdot\nabla\zeta
-G\Delta\zeta
-Gu\cdot\nabla\zeta.
\end{align*}
It follows from~\eqref{eq:head-Gzeta-start} that
\begin{align}
\int_{B_4}\delta_{h,x_0}H\,dx
\leq{}&
\left|
\int_{B_4}
H(u-u_k)\cdot\nabla G\,\zeta\,dx
\right|+
2\left|
\int_A
H\nabla G\cdot\nabla\zeta\,dx
\right| \nonumber \\
&+
\left|
\int_A
HG\Delta\zeta\,dx
\right|
+
\left|
\int_A
HGu\cdot\nabla\zeta\,dx
\right|.
\label{eq:head-green-detailed}
\end{align}
The annular terms are bounded by
\begin{align*}
2\left|
\int_AH\nabla G\cdot\nabla\zeta\,dx
\right|
+
\left|
\int_AHG\Delta\zeta\,dx
\right|
\leq
C\|H\|_{L^2(A)}
\left(
\|\nabla G\|_{L^2(A)}
+\|G\|_{L^2(A)}
\right),
\end{align*}
while
\[
\left|
\int_A
HGu\cdot\nabla\zeta\,dx
\right|
\leq
C
\|H\|_{L^2(A)}
\|u\|_{L^4(A)}
\|G\|_{L^4(A)}.
\]
By property~(P3), these quantities are bounded uniformly in
$h,k,x_0$.

For the remaining  terms, property~(P2) gives, for each
fixed $h>0$,
\[
\begin{aligned}
\left|
\int_{B_4}
H(u-u_k)\cdot\nabla G\,\zeta\,dx
\right|
\leq
\|H\|_{L^2(B_4)}
\|u-u_k\|_{L^4(B_4)}
\|\nabla G\|_{L^4(B_4)}
\longrightarrow0
\quad (k\to\infty),
\end{aligned}
\]
uniformly for $x_0\in B_1$.
Letting $k\to\infty$ for each fixed $h$ gives
\[
 \int_{B_4}\delta_{h,x_0}H\,dx\leq C(n,M),
\]
with a constant independent of $h$ and $x_0$. For each fixed
$x_0\in B_1\setminus\{x'=0\}$, the function $H$ is smooth near
$x_0$. Letting $h\downarrow0$ therefore gives
$H(x_0)\leq C(n,M)$ at every such point.
Since the axis $\{x'=0\}$ has Lebesgue measure zero, this gives
\begin{equation}
\label{eq:Hbound}
\|H_+\|_{L^\infty(B_1)}
\leq C(n,M).
\end{equation}

\subsection{From \texorpdfstring{$H_+\in L^\infty$}{bounded positive head pressure} to regularity of \texorpdfstring{$u$}{u}}\label{sec:regularity}
	A crucial step in our argument is the deduction of uniform bounds for $u$
	and $\nabla u$ from the $L^\infty$ estimate for $H_+$.  The following
	localized weighted estimate is the key ingredient underlying the
	regularity criterion of Frehse-R\r{u}\v{z}i\v{c}ka
	\cite[Theorem~1.8]{FrehseRuzicka94}.  We include a localized proof in the
	appendix in the form needed here.

	\begin{proposition}[Interior weighted estimate]
		\label{prop:FR-interior-weighted}
		Let \(n\geq5\), let \(\mathcal O\subset\mathbb R^n\) be open, and suppose that $B_{4R_*}(z)\Subset\mathcal O.$
		Let $u\in W^{1,2}(B_{4R_*}(z);\mathbb R^n),$ $
		p\in W^{1,n/(n-1)}(B_{4R_*}(z))$
		be a weak solution of
		\[
		-\Delta u+(u\cdot\nabla)u+\nabla p=0,
		\quad
		\operatorname{div}u=0
		\quad\text{in }B_{4R_*}(z).
		\]
		Assume, in addition, that
		\[ H\in L^2(B_{4R_*}(z)), \quad
		H_+\in L^q(B_{4R_*}(z))
		\quad\text{for some }q>\frac n2,
		\]
		and that the local Navier-Stokes inequality
		\begin{equation}
			\int_{B_{4R_*}(z)}
			\nabla u:\nabla(u\varphi)\,dx
			\leq
			\int_{B_{4R_*}(z)}
			H\,u\cdot\nabla\varphi\,dx
			\label{eq:FR-local-NS-inequality}
		\end{equation}
		holds for every nonnegative
		$
		\varphi\in C_c^\infty(B_{4R_*}(z)).
		$

		Then there exist constants
		$
		\beta\in(0,1),\, C<\infty,
		$
		depending only on
		\[
		n,\quad q,\quad R_*,
		\quad
		\|u\|_{W^{1,2}(B_{4R_*}(z))},
		\quad
		\|p\|_{W^{1,n/(n-1)}(B_{4R_*}(z))},
		\quad
		\|H_+\|_{L^q(B_{4R_*}(z))},
		\]
		such that
		\begin{equation}
			\int_{B_\rho(x_0)}
			\frac{|\nabla u(x)|^2}
			{|x-x_0|^{n-4}}\,dx
			\leq C\rho^\beta
			\label{eq:FR-interior-weighted-decay}
		\end{equation}
		for every $
		x_0\in B_{R_*}(z), \
		0<\rho<\frac{R_*}{4}.
		$
	\end{proposition}

	The proof follows the weighted argument in
	\cite[Sections~2-3]{FrehseRuzicka94}; the details are given in
	Appendix~\ref{sec:appendix}.
	We now apply Proposition~\ref{prop:FR-interior-weighted} with
	$z=0$, $R_*=\frac14$, and $q=n$. Lemma~\ref{lem:weak-extension}
	provides all the required Sobolev bounds and the local energy inequality,
	while \eqref{eq:Hbound} gives
	\[
	 H_+\in L^\infty(B_1),
	 \quad \|H_+\|_{L^\infty(B_1)}\le C(n,M).
	\]
	Hence there exist $\beta\in(0,1)$ and $C=C(n,M)$ such that
	\begin{equation}\label{eq:grad-morrey-decay}
	 \int_{B_\rho(x_0)}|\nabla u|^2\,dx
	 \le C\rho^{n-4+\beta}
	\end{equation}
	for every $x_0\in B_{1/4}$ and $0<\rho<1/16$.

	Set
	\begin{equation}\label{eq:morrey-lambda-d}
	 \lambda:=n-4+\beta,
	 \quad d:=n-\lambda=4-\beta\in(3,4).
	\end{equation}
	For a ball $D\Subset B_{1/4}$, we use the standard Morrey norm
	\[
	 \|f\|_{L^{t,\lambda}(D)}^t
	 :=\sup_{x\in D,\,0<\rho\le \operatorname{diam}D}
	 \rho^{-\lambda}\int_{B_\rho(x)\cap D}|f|^t\,dx.
	\]
	Together with the global $L^2$ bound from
	Lemma~\ref{lem:weak-extension}, \eqref{eq:grad-morrey-decay} yields
	\begin{equation}\label{eq:grad-morrey-start}
	 \nabla u\in L^{2,\lambda}_{\rm loc}(B_{1/4}),
	 \quad
	 \|\nabla u\|_{L^{2,\lambda}(B_{7/32})}\le C(n,M).
	\end{equation}

	We shall use two standard local estimates in Morrey spaces.  First,
	the proof of Theorem~A.2 of Li-Yang~\cite{LiYang22}, after scaling and
	localization, yields the following form on nested balls $D'\Subset D$:
	if $f\in L^1(D)$ and $\nabla f\in L^{a,\lambda}(D)$ with
	$1<a<\infty$, then
	\begin{equation}\label{eq:morrey-sobolev-local}
	 \|f\|_{L^{a^*,\lambda}(D')}
	 \le C\Bigl(
	 \|\nabla f\|_{L^{a,\lambda}(D)}+
	 \|f\|_{L^1(D)}\Bigr),
	\end{equation}
	where, with $d=n-\lambda$,
	\begin{equation}\label{eq:morrey-sobolev-exponent}
	 \frac1{a^*}=\frac1a-\frac1d
	 \quad\text{if }1<a<d,
	\end{equation}
	while if $a\ge d$, the exponent $a^*$ may be chosen arbitrarily large
	but finite.  We shall also use the following  version of the
	local Stokes-Morrey estimate.

	\begin{lemma}[Local Stokes-Morrey estimate]\label{lem:weak-stokes-morrey}
	Let $D'\Subset D$ be balls, $1<r<\infty$, and $0\le\lambda<n$.
	Suppose that, in the sense of distributions on $D$,
	\[
	 -\Delta v+\nabla\pi=F,\quad \operatorname{div}v=0,
	\]
	where $F\in L^{r,\lambda}(D)$ and $v\in L^1(D)$.  Then
	\begin{equation}\label{eq:weak-stokes-morrey}
	 \|\nabla^2v\|_{L^{r,\lambda}(D')}
	 +\|\nabla\pi\|_{L^{r,\lambda}(D')}
	 \le C\Bigl(
	 \|F\|_{L^{r,\lambda}(D)}+\|v\|_{L^1(D)}\Bigr),
	\end{equation}
	where $C$ depends only on $n,r,\lambda$ and the two nested balls;
	the pressure is understood modulo an additive constant.
	\end{lemma}

	\begin{proof}
	Choose intermediate balls $D'\Subset D_1\Subset D_2\Subset D$ and
	mollify the distributional Stokes system inside $D_2$.  The mollified
	pair $(v_\varepsilon,\pi_\varepsilon)$ satisfies the smooth Stokes system
	with right-hand side $F_\varepsilon$ on $D_1$.  Convolution is bounded in
	Morrey spaces on interior subdomains, so
	\[
	 \|F_\varepsilon\|_{L^{r,\lambda}(D_1)}
	 \le C\|F\|_{L^{r,\lambda}(D)},
	 \quad
	 \|v_\varepsilon\|_{L^1(D_1)}\le C\|v\|_{L^1(D)}.
	\]
	After covering $D'$ by finitely many balls and rescaling
	Theorem~A.3 of Li-Yang~\cite{LiYang22}, we obtain
	\eqref{eq:weak-stokes-morrey} for the mollified pair with a constant
	independent of $\varepsilon$.  Since $r>1$, weak compactness in
	$L^r_{\rm loc}$, identification of the distributional derivatives, and
	lower semicontinuity on each ball give \eqref{eq:weak-stokes-morrey} in
	the limit $\varepsilon\downarrow0$.
	\end{proof}

\begin{proof}[Proof of Theorem \ref{thm:local-regularity}]
We complete the proof by a finite bootstrap in Morrey spaces.
Assume that, on a ball
	$D\Subset B_{1/4}$,
	\begin{equation}\label{eq:morrey-induction-hyp}
	 \nabla u\in L^{t,\lambda}(D),
	 \quad 2\le t<d,
	\end{equation}
	with quantitative control.  Applying
	\eqref{eq:morrey-sobolev-local} to $u$ on a slightly smaller ball gives
	\begin{equation}\label{eq:morrey-u-s}
	 u\in L^{s,\lambda},
	 \quad
	 \frac1s=\frac1t-\frac1d.
	\end{equation}
	At each step we combine this improved integrability of $u$ with the
	fixed starting estimate
	$\nabla u\in L^{2,\lambda}$ from
	\eqref{eq:grad-morrey-start}.  H\"older's inequality on each ball then
	gives that
	\begin{equation}\label{eq:morrey-force-r}
	 (u\cdot\nabla)u\in L^{r,\lambda},
	 \quad
	 \frac1r=\frac1s+\frac12
	 =\frac1t+\frac12-\frac1d.
	\end{equation}
	Indeed, the powers of the radius match because
	$r/s+r/2=1$.  Since $t<d$ and $d>3$, we have $1<r<2<d$.
	By H\"older's inequality on balls, the  estimate
	$\nabla u\in L^{2,\lambda}$ also implies
	$\nabla u\in L^{r,\lambda}$ locally for every $r<2$; hence the
	Morrey-Sobolev estimate \eqref{eq:morrey-sobolev-local} may be applied to $\nabla u$ after the use of the Stokes
	estimate below.  Applying Lemma~\ref{lem:weak-stokes-morrey} to
	\[
	 -\Delta u+\nabla p=-(u\cdot\nabla)u
	\]
	and then applying \eqref{eq:morrey-sobolev-local} to $\nabla u$ yields
	\begin{equation}\label{eq:morrey-improvement}
	 \nabla u\in L^{t_+,\lambda},
	 \quad
	 \frac1{t_+}=\frac1r-\frac1d
	 =\frac1t+\frac12-\frac2d
	 =\frac1t-\frac{\beta}{2(4-\beta)}.
	\end{equation}
	Thus every subcritical step decreases the reciprocal exponent by the same
	positive amount
	\[
	 \delta:=\frac{\beta}{2(4-\beta)}.
	\]

	Starting from $t_0=2$, define, as long as $t_j<d$,
	\[
	 \frac1{t_{j+1}}=\frac1{t_j}-\delta.
	\]
	After at most $\lceil(2-\beta)/\beta\rceil$ subcritical steps,
	we reach $t_j\geq d$. Indeed,
	$1/t_j=1/2-j\delta$ as long as the preceding exponent is below $d$.
	The new reciprocal remains positive at the final step because
	$1/d-\delta=(2-\beta)/(2d)>0$.  To keep the constants
	quantitative, fix in advance a number $N=N(\beta)$ large enough to
	accommodate all the Morrey-Sobolev and Stokes localizations in this finite
	iteration, and choose nested concentric balls
	\[
	 B_{5/32}\Subset D_N\Subset D_{N-1}\Subset\cdots\Subset D_0
	 \Subset B_{7/32}
	\]
	with equal positive gaps between consecutive radii.  All localization
	constants then depend only on $N$ and hence, through $\beta$, only on the
	data of the local theorem.

	If some $t_j>d$, the desired supercritical Morrey exponent has already
	been reached.  If instead $t_j=d$, the critical case of
	\eqref{eq:morrey-sobolev-local} gives
	$u\in L^{s,\lambda}_{\rm loc}$ for every finite $s$.  Choose
	\[
	 s>\frac{2d}{4-d}=\frac{2d}{\beta},
	 \quad
	 \frac1r=\frac1s+\frac12.
	\]
	Then $1<r<2<d$ and
	\[
	 \frac1r-\frac1d<\frac1d.
	\]
	One final application of Lemma~\ref{lem:weak-stokes-morrey} and
	\eqref{eq:morrey-sobolev-local} therefore gives, in either case,
	\begin{equation}\label{eq:morrey-supercritical}
	 \nabla u\in L^{t_*,\lambda}(B_{5/32})
	 \quad\text{for some }t_*>d=n-\lambda,
	\end{equation}
	with norm bounded by a constant depending only on $n$ and $M$.

	Since $t_*+\lambda>n$, the Morrey-Campanato embedding applied to
	\eqref{eq:morrey-supercritical} gives
	\begin{equation}\label{eq:u-holder-bound}
	 \|u\|_{C^{0,\alpha_*}(B_{1/8})}
	 \le C(n,M),
	 \quad
	 \alpha_*=1-\frac{d}{t_*}>0.
	\end{equation}
	In particular, $u$ is uniformly bounded on $B_{1/8}$.

	It remains to apply standard local Stokes regularity.  First write
	\[
	 -\Delta u+\nabla p=-\operatorname{div}(u\otimes u).
	\]
	Since $u\in L^\infty(B_{1/8})$, one has
	$u\otimes u\in L^q(B_{1/8})$ for every finite $q$.  The interior Stokes
	estimate in divergence form gives, for every finite $q$,
	\[
	 \|\nabla u\|_{L^q(B_{3/32})}
	 +\|p-(p)_{B_{3/32}}\|_{L^q(B_{3/32})}
	 \le C(q,n,M).
	\]
	Consequently $(u\cdot\nabla)u\in L^q(B_{3/32})$.  Applying the
	non-divergence-form Stokes estimate on $B_{1/16}$ gives
	\[
	 \|u\|_{W^{2,q}(B_{1/16})}
	 +\|p-(p)_{B_{1/16}}\|_{W^{1,q}(B_{1/16})}
	 \le C(q,n,M).
	\]
	Choose $q>n$.  Sobolev embedding then yields
	\begin{equation}\label{Bound}
	 \|u\|_{L^\infty(B_{1/16})}
	 +\|\nabla u\|_{L^\infty(B_{1/16})}
	 \le C_2(n,M).
	\end{equation}
Further applications of interior Stokes estimates imply that
$u,p\in C^\infty(B_{1/16})$. Thus the singularity on the axis is removable
in $B_{1/16}$. This proves the working-scale estimate
\eqref{eq:B8-working-scale}. Theorem~\ref{thm:local-regularity} now follows
from the fixed scaling described in Remark~\ref{rem:B8-working-scale}.	
\end{proof}

	\section{Global rigidity}\label{sec:proof}
	\begin{proof}[Proof of Theorem~\ref{thm:main}]

	Let $x\in\R^n\setminus\{x'=0\}$ be arbitrary, and let
	$R>16|x|$.  Define the rescaled functions
	\[
	v(y)=Ru(Ry),\quad p_v(y)=R^2p(Ry),\quad
	y\in\R^n\setminus\{y'=0\}.
	\]
	One verifies directly that $(v,p_v)$ satisfies~\eqref{NS}
	in $\R^n\setminus\{y'=0\}$ with
	\[
	|v(y)|=R\,|u(Ry)|\leq\frac{C_1R}{|Ry'|}=\frac{C_1}{|y'|}.
	\]
	Thus $v$ satisfies the same hypotheses as $u$.
	By the working-scale estimate \eqref{eq:B8-working-scale},
\[
 |v(y)|\leq C_2(n,C_1),
 \qquad y\in B_{1/16}.
\]
	In particular, taking $y=x/R\in B_{1/16}$,
	\[
	|u(x)|=R^{-1}|v(x/R)|\leq \frac{C_2}{R}.
	\]
	Since $R$ may be arbitrarily large, letting $R\to\infty$ gives
	$u(x)=0$.  As $x$ was arbitrary, $u\equiv0$ in
	$\R^n\setminus\{x'=0\}$.
	\end{proof}


	\section{Applications of the Main theorem}\label{sec:applications}
	\subsection{Removable singularities along a line}
	Theorem~\ref{thm:local-regularity} immediately gives the following local
	removability result.

	\begin{theorem}\label{thm:removable}
		Let $n\geq5$, let $\Omega\subset\R^n$ be open, and let
		$\Sigma\subset\Omega\cap\{x'=0\}$ be relatively closed in $\Omega$.
		Suppose that $(u,p)$ is a smooth solution of~\eqref{NS} in
		$\Omega\setminus\Sigma$.  Assume that every point $x_0\in\Sigma$ has a
		neighbourhood $U\Subset\Omega$ and a constant $C_U$ such that
		\begin{equation}\label{eq:removablecond}
			|u(x)|\leq \frac{C_U}{\dist(x,\Sigma)},
			\quad x\in U\setminus\Sigma.
		\end{equation}
		Then $(u,p)$ extends across $\Sigma$ to a smooth solution of
		\eqref{NS} in $\Omega$.
	\end{theorem}

	\begin{proof}
		Fix $x_0\in\Sigma$ and translate the coordinates so that $x_0=0$.
		Choose $R>0$ so small that $B_{8R}\Subset U$.  Since
		$\Sigma\subset\{x'=0\}$, for every $x\in B_{8R}\setminus\{x'=0\}$,
		\[
		 \dist(x,\Sigma)\ge |x'|.
		\]
		Hence \eqref{eq:removablecond} implies the transverse estimate
		\[
		 |u(x)|\le \frac{C_U}{|x'|}
		 \quad\text{in }B_{8R}\setminus\{x'=0\}.
		\]
		(The solution is smooth there because
		$B_{8R}\setminus\{x'=0\}\subset\Omega\setminus\Sigma$.)
		Define
		\begin{equation}\label{eq:rescaled}
			\tilde u(y)=R\,u(Ry),\quad
			\tilde p(y)=R^2\,p(Ry),
			\quad y\in B_8\setminus\{y'=0\}.
		\end{equation}
		Then $(\tilde u,\tilde p)$ satisfies \eqref{NS} and
		\begin{equation}\label{eq:rescaledbound}
			|\tilde u(y)|\leq\frac{C_U}{|y'|},
			\quad y\in B_8\setminus\{y'=0\}.
		\end{equation}
		By Remark~\ref{rem:B8-working-scale}, $(\tilde u,\tilde p)$ extends
smoothly to $B_{1/16}$ and satisfies
		\begin{equation}\label{eq:localbounded}
			\|\tilde u\|_{L^\infty(B_{1/16})}
			+\|\nabla\tilde u\|_{L^\infty(B_{1/16})}\le C(n,C_U).
		\end{equation}
		Scaling back, $(u,p)$ extends smoothly to $B_{R/16}(x_0)$.
		The local extensions agree on overlaps: they agree with the original
		pair off the axis, and hence everywhere by continuity. Since
		$x_0\in\Sigma$ was arbitrary, the singularity on   $\Sigma$ is removable.
	\end{proof}

\subsection{Asymptotic behavior in exterior domains}
\label{sec:exterior}

We next apply the local estimate established above to stationary
Navier-Stokes flows in exterior domains.  The main point is that the
transverse critical bound can be upgraded to
the isotropic critical bound.

We use the optimal exterior decay theorem of
Bang-Gui-Liu-Wang-Xie~\cite{BangGuiLiuWangXie} and record it here for convenience. 
\begin{theorem}[Exterior asymptotics {\cite[Theorem 1.3]{BangGuiLiuWangXie}}]\label{thm:BGLWX_decay}
	For $n\geq 5$,    let $u$ be a smooth solution to the Navier-Stokes equations \eqref{NS} in $\mathbb{R}^n \setminus \overline{B}_1$. Assume $u$  satisfies
	\begin{equation}\label{corollary1-condition_2}
		|u(x)|\leq \frac{C}{|x|} \quad \text{for }\,\, x\in \mathbb{R}^n\setminus \overline{B}_1
	\end{equation}
	for some constant $C>0$.
	Then  as  $|x|\to \infty$, we have
	\begin{equation}\label{eq:BGLWX-expansion}
		u_i (x) =
		b_j G_{ij}(x)
		+O\left(\frac{1}{|x|^{n-1}}\right).
	\end{equation}
	In \eqref{eq:BGLWX-expansion},
	$G_{ij}$ is the Green tensor of the  Stokes equations given by
	\begin{equation}\label{eq:Greentensor}
		G_{ij}(x)= \frac{1}{2n\omega_n} \left[\frac{\delta_{ij}}{(n-2)|x|^{n-2}} +\frac{ x_i x_j}{|x|^n}\right],
	\end{equation}
	and $b=(b_1,b_2,\cdots,b_n)$ is defined by
	\begin{equation}\label{defb}
		b_i= \int_{\partial B_{R}} T_{ij}(u,p)\nu_j\,d\sigma,
		\quad i=1,2,\ldots,n, \quad R>1.
	\end{equation}
	Here $\omega_n$ is the volume of the unit ball in $\mathbb{R}^n$,
	\begin{equation}
		T_{ij}(u, p) = p\delta_{ij} + u_i u_j - \left( \frac{\partial u_i}{\partial x_j} + \frac{\partial u_j}{\partial x_i}\right)
	\end{equation}
	is the momentum flux density tensor of the fluid, and $\nu_j$ denotes the $j$-th component of the unit outer normal vector $\nu$ on $\partial B_{R}$.
	Here $\partial_jT_{ij}=0$ in the exterior region, so the flux is
	independent of $R$.
\end{theorem}

\begin{theorem}[Isotropic decay and asymptotics at infinity]
	\label{thm:exterior-optimal-decay}
	Let $n\geq5$, and let $\Omega\subset\mathbb R^n$ be an exterior domain;
	that is, $\mathbb R^n\setminus\Omega$ is compact.  Suppose that $(u,p)$ is
	a smooth solution of~\eqref{NS} in $\Omega\setminus \{x'=0\}$.  Assume that there exist
	$R_0>0$ and $C_1>0$ such that
	\begin{equation}\label{eq:exterior-transverse}
		|u(x)|\leq\frac{C_1}{|x'|}
		\quad
		\text{for }x\in\Omega,\quad |x|\geq R_0,\quad x'\neq0.
	\end{equation}
	Then the singularity on the axis is removable outside a sufficiently large ball.
Moreover, there exist $R_1\geq R_0$ and
	$C=C(n,C_1)>0$ such that
	\begin{equation}\label{eq:exterior-isotropic}
		|u(x)|\leq\frac{C}{|x|},
		\quad |x|\geq R_1.
	\end{equation}
	Consequently, the expansion~\eqref{eq:BGLWX-expansion} holds as
	$|x|\to\infty$.  In particular,
	\[
	|u(x)|=O(|x|^{2-n}).
	\]
	If $b=0$, then
	\[
	|u(x)|=O(|x|^{1-n}).
	\]
\end{theorem}

\begin{proof}
	Since $\mathbb R^n\setminus\Omega$ is compact, after increasing $R_0$ if
	necessary, we may assume that
	\begin{equation}\label{eq:exterior-contains-ball-complement}
		\mathbb R^n\setminus \overline{B}_{R_0}\subset\Omega.
	\end{equation}
	We first show the singularity on  the axis is removable in this exterior region.  Fix a point
	$x_0\in\{x'=0\}$ with $|x_0|>R_0$.  Choose $r=r(x_0)>0$ so small
	that $B_{8r}(x_0)\Subset\mathbb R^n\setminus\overline{B}_{R_0}$.
	After translating by $x_0$ and rescaling by $r$, the transverse estimate
\eqref{eq:exterior-transverse} is exactly of the working-scale form in
Remark~\ref{rem:B8-working-scale}.  Hence $(u,p)$ extends smoothly
	across the axis in a neighborhood of $x_0$.  Since $x_0$ was arbitrary,
	the whole axis portion in $\mathbb R^n\setminus\overline{B}_{R_0}$ is
	removable, and $(u,p)$ is a smooth solution throughout that exterior
	region.

	Fix $R\geq 4R_0$ and define
	\begin{equation}\label{eq:exterior-rescaling}
		u_R(y)=R u(Ry),
		\quad
		p_R(y)=R^2p(Ry).
	\end{equation}
	Then $(u_R,p_R)$ solves~\eqref{NS} in $\mathbb{R}^n \setminus \overline{B}_{R_0/R}$.  Let \(
	A:=B_2\setminus\overline{B_{1/2}}.
	\) In view of
	\eqref{eq:exterior-contains-ball-complement}, $A\Subset \mathbb{R}^n \setminus \overline{B}_{R_0/R}$.
The transverse bound is invariant
	under the scaling. Hence,
	\begin{equation}\label{eq:scaled-transverse}
		|u_R(y)|
		=R|u(Ry)|
		\leq\frac{C_1R}{|(Ry)'|}
		=\frac{C_1}{|y'|},
		\quad y\in A,\quad y'\neq0.
	\end{equation}
	We claim that
	\begin{equation}\label{eq:uniform-sphere-bound}
		\|u_R\|_{L^\infty(\partial B_1)}\leq C(n,C_1)
	\end{equation}
	with a constant independent of $R\geq4R_0$.

	Let
	\[
	e_n=(0,\ldots,0,1),
	\quad -e_n=(0,\ldots,0,-1)
	\]
	be the two points at which the axis intersects with $\partial B_1$. Choose a fixed $\rho_0>0$ sufficiently small so that
\(
B_{8\rho_0}(e_n)\cup B_{8\rho_0}(-e_n)
\Subset B_2\setminus\overline{B_{1/2}}.
\)  If $y\in\partial B_1$ lies outside
	$B_{\rho_0/16}(e_n)\cup B_{\rho_0/16}(-e_n)$, then
	$|y'|\geq c\rho_0$.
	Consequently,~\eqref{eq:scaled-transverse} gives
	\begin{equation}\label{eq:away-poles}
		|u_R(y)|\leq \frac{C_1}{c\rho_0}
	\end{equation}
	uniformly in $R$.

	It remains to control $u_R$ near $e_n$ and $-e_n$.  The balls
	$B_{8\rho_0}(\pm e_n)$ are compactly contained in $A$.  After translating by
	either $e_n$ or $-e_n$ and rescaling by the fixed factor $\rho_0$, the pair $(u_R,p_R)$ satisfies the hypotheses of the working-scale estimate in
Remark~\ref{rem:B8-working-scale}. Therefore
	\begin{equation}\label{eq:pole-bounds}
		\|u_R\|_{L^\infty(B_{\rho_0/16}(e_n))}
		+\|u_R\|_{L^\infty(B_{\rho_0/16}(-e_n))}
		\leq C(n,C_1).
	\end{equation}
	The constant is independent of $R$, because both the equations and the
	bound~\eqref{eq:scaled-transverse} are invariant under the
	Navier-Stokes scaling, while $\rho_0$ is fixed.

	Combining~\eqref{eq:away-poles} and~\eqref{eq:pole-bounds} proves
	\eqref{eq:uniform-sphere-bound}.  Hence, for every $x$ with
	$|x|=R\geq4R_0$, taking $y=x/R\in\partial B_1$ gives
	\[
	R|u(x)|=|u_R(y)|\leq C(n,C_1).
	\]
	Since $R$ is arbitrary, we obtain
	\[
	|u(x)|\leq\frac{C(n,C_1)}{|x|}
	\quad\text{for all }|x|\geq4R_0.
	\]
	Thus~\eqref{eq:exterior-isotropic} holds, for instance with
	$R_1=4R_0$.

	Finally, restrict $(u,p)$ to
	$\mathbb R^n\setminus\overline{B}_{R_1}$ and rescale this region to
	$\mathbb R^n\setminus\overline{B}_1$.  Theorem~\ref{thm:BGLWX_decay} applies and
	yields~\eqref{eq:BGLWX-expansion}.  Since
	$G_{ij}(x)=O(|x|^{2-n})$, the expansion implies
	$|u(x)|=O(|x|^{2-n})$; when $b=0$, its remainder gives
	$|u(x)|=O(|x|^{1-n})$.
\end{proof}

	\appendix

	\section{Proof of Proposition  \ref{prop:FR-interior-weighted} }
	\label{sec:appendix}
We first recall the fixed-scale weighted estimate that will be used
	below. This estimate is the local analogue of the weighted estimate
	in \cite{FrehseRuzicka94}. Since the hypotheses in the present
	problem are slightly different from those in that reference, we use
	the following local form.

	\begin{lemma}[Fixed-scale weighted estimate]
		\label{lem:FR-fixed-scale-H}
		Under the assumptions of
		Proposition~\ref{prop:FR-interior-weighted}, there exists a constant
		\(C\), depending only on \(n\), \(R_*\), $q$, and the local norms
		\[
		\|u\|_{W^{1,2}(B_{4R_*}(z))},\quad
		\|p\|_{W^{1,\frac n{n-1}}(B_{4R_*}(z))},\quad
		\|H_+\|_{L^q(B_{4R_*}(z))},
		\]
		such that, with $\mathcal B:=B_{3R_*}(z)$,
		\begin{equation}
			\sup_{x_0\in \mathcal B}
			\int_{\mathcal B}
			\frac{|H(x)|}{|x-x_0|^{n-2}}\,dx
			\leq C.
			\label{eq:FR-fixed-scale-H}
		\end{equation}
	\end{lemma}

	\begin{proof}
		Fix \(x_0\in\mathcal B=B_{3R_*}(z)\). Choose
		$$
		\eta\in C_c^\infty(B_{4R_*}(z)),
		\quad
		0\leq\eta\leq1,
		\quad
		\eta\equiv1\quad\text{on }B_{\frac72R_*}(z),
		$$
		with
		$$
		|\nabla\eta|\leq CR_*^{-1},
		\quad
		|D^2\eta|\leq CR_*^{-2}.
		$$

		For \(0<h<R_*\), let
		$$
		r:=|x-x_0|,
		\quad
		w_h(x):=(r^2+h^2)^{-(n-4)/2},
		\quad
		\varphi_h:=\eta^2w_h.
		$$
		We use \(\nabla\varphi_h\) as a test function in the 
		Navier-Stokes equations. Since \(\operatorname{div}u=0\),
			\begin{equation}\label{eq:fixed-H-test}
			\int_{B_{4R_*}(z)}
			p\Delta\varphi_h\,dx
			+
			\int_{B_{4R_*}(z)}
				u_i u_j\partial_{ij}\varphi_h\,dx
			=0.
		\end{equation}
		A direct computation gives
		$$
		\partial_{ij}w_h
		=
		-(n-4)\delta_{ij}(r^2+h^2)^{-(n-2)/2}
		+
		(n-4)(n-2)
		\frac{(x_i-x_{0,i})(x_j-x_{0,j})}
		{(r^2+h^2)^{n/2}},
		$$
		and
		$$
		\Delta w_h
		=
		-2(n-4)(r^2+h^2)^{-(n-2)/2}
		-(n-4)(n-2)
		\frac{h^2}{(r^2+h^2)^{n/2}}.
		$$
		Set
		\[
		 A_h=(r^2+h^2)^{-(n-2)/2},\quad
		 B_h=(r^2+h^2)^{-n/2}.
		\]
		The terms in~\eqref{eq:fixed-H-test} in which both derivatives fall
		on $w_h$ are
		\begin{align*}
		 &-2(n-4)\int_{B_{4R_*}(z)} H\eta^2A_h\, dx
		 +(n-4)(n-2)\int_{B_{4R_*}(z)} |u\cdot(x-x_0)|^2\eta^2B_h\, dx \\
		 &\hspace{35mm}-(n-4)(n-2)\int_{B_{4R_*}(z)} p\eta^2h^2B_h\, dx.
		\end{align*}
		 Using
		\[
		 |H|=2H_+-H,\quad p=H-\frac12|u|^2,
		\]
		we rewrite the preceding expression as
		\begin{align*}
		 &2(n-4)\int_{B_{4R_*}(z)} |H|\eta^2A_h\, dx
		 +(n-4)(n-2)\int_{B_{4R_*}(z)} |u\cdot(x-x_0)|^2\eta^2B_h\, dx\\
		 &\quad +(n-4)(n-2)\int_{B_{4R_*}(z)} |H|\eta^2h^2B_h \, dx
		 +\frac{(n-4)(n-2)}2\int_{B_{4R_*}(z)} |u|^2\eta^2h^2B_h\, dx\\
		 &\quad -4(n-4)\int_{B_{4R_*}(z)} H_+\eta^2A_h \, dx
		 -2(n-4)(n-2)\int_{B_{4R_*}(z)} H_+\eta^2h^2B_h\, dx.
		\end{align*}
		The remaining terms in \eqref{eq:fixed-H-test} contain  derivatives of $\eta$.  Moving them and
		the terms containing two $H_+$  to the right, and discarding the two nonnegative
		terms containing $h^2$, gives
		\begin{align}
			&c_n\int_{B_{4R_*}(z)}
			\frac{|H|\eta^2}{(r^2+h^2)^{(n-2)/2}}\,dx
			+
			c_n\int_{B_{4R_*}(z)}
			\frac{|u\cdot(x-x_0)|^2\eta^2}
			{(r^2+h^2)^{n/2}}\,dx
			\notag\\
			&\quad\leq
			C\int_{\operatorname{supp}\nabla\eta}
			\frac{|p|+|u|^2}
			{(r^2+h^2)^{(n-2)/2}}\,dx
			+
			C\int_{B_{4R_*}(z)}
			\frac{H_+\eta^2}
			{(r^2+h^2)^{(n-2)/2}}\,dx.
						\label{eq:FR-fixed-scale-identity}
		\end{align}

		Since
		\[
		\operatorname{dist}
		\bigl(x_0,\operatorname{supp}\nabla\eta\bigr)
		\geq \frac{R_*}{2},
		\]
		the terms involving \(\nabla\eta\) or \(D^2\eta\) are bounded by
		\[
		C(R_*)\left(
		\|p\|_{L^1(B_{4R_*}(z))}
		+
		\|u\|_{L^2(B_{4R_*}(z))}^2
		\right),
		\]
		uniformly in \(x_0\in B_{3R_*}(z)\) and \(0<h<R_*\).
Moreover, \(H_+\in L^q(B_{4R_*}(z))\) for some \(q>n/2\),
		H\"older's inequality gives
		\begin{align}
			\int_{B_{4R_*}(z)}
			\frac{H_+}
			{(r^2+h^2)^{(n-2)/2}}\,dx
			&\leq
			\|H_+\|_{L^q(B_{4R_*}(z))}
			\left\|
			(r^2+h^2)^{-(n-2)/2}
			\right\|_{L^{q'}(B_{4R_*}(z))}
			\leq C,
			\label{eq:FR-fixed-scale-Hplus}
		\end{align}
		where
		\(
		q'=\frac{q}{q-1}<\frac{n}{n-2}.
		\)
		Thus the right-hand side of
		\eqref{eq:FR-fixed-scale-identity} is bounded independently of
		\(h\) and \(x_0\). Hence
		\[
		\int_{B_{4R_*}(z)}
		\frac{|H|\eta^2}
		{(r^2+h^2)^{(n-2)/2}}\,dx
		\leq C.
		\]
		Since \(\eta\equiv1\) on \(\mathcal B=B_{3R_*}(z)\),
		Fatou's lemma yields
		\[
		\int_{\mathcal B}
		\frac{|H(x)|}{|x-x_0|^{n-2}}\,dx
		\leq C.
		\]
		Taking the supremum over
		\(x_0\in \mathcal B\) proves
		\eqref{eq:FR-fixed-scale-H}.  
	\end{proof}

We now prove Proposition~\ref{prop:FR-interior-weighted}.

\begin{proof}[Proof of Proposition \ref{prop:FR-interior-weighted}]
The proof is divided into six steps. All integrations involving
$g$ below are first performed at fixed $h>0$. Since
$g\in W^{2,2}\cap L^\infty$, $u\in W^{1,2}$ and $H\in L^2$,
the functions $u_i\partial_i(\eta^2w_h)$ and
$(g-g_R)\eta^2v_h$ belong to $W^{1,2}_0$ on their support.
They are admissible test functions for the Poisson equation by density. Identities
involving $g^2$ follow by approximation in $W^{1,1}$. Thus no
smoothness across the axis is used in these weighted calculations.

{\it Step 1. Construction of the auxiliary function.}

	Set
	\[
	\mathcal B:=B_{3R_*}(z),
	\quad
	H:=\frac12|u|^2+p.
	\]
	By the assumptions of
	Proposition~\ref{prop:FR-interior-weighted},
	\[
	u\in W^{1,2}(B_{4R_*}(z)),
	\quad
	p\in W^{1,\frac n{n-1}}(B_{4R_*}(z)),
	\quad
	H\in L^2(B_{4R_*}(z)), \quad H_{+}\in L^q(B_{4R_*}(z)), \quad q>\frac{n}{2}.
	\]

	We now construct the auxiliary function. Let
	\(g\in W^{1,2}_0(\mathcal B)\) be the unique weak solution of
	\begin{equation}
		\begin{cases}
			-\Delta g=H &\text{in }\mathcal B,\\
			g=0 &\text{on }\partial\mathcal B.
		\end{cases}
		\label{eq:FR-local-g}
	\end{equation}
	The existence and uniqueness of \(g\) follow from the Lax-Milgram
	theorem. Since \(H\in L^2(\mathcal B)\), standard elliptic regularity
	gives
	\(
	g\in W^{2,2}(\mathcal B).
	\)

	Let \(G_{\mathcal B}(x,y)\) denote the Dirichlet Green function of
	\(-\Delta\) in \(\mathcal B\). Since \(\mathcal B\) is a ball,
	\[
	0\leq G_{\mathcal B}(x,y)
	\leq C|x-y|^{2-n},
	\quad
	x,y\in\mathcal B,\quad x\neq y.
	\]
	The Green representation formula gives
	\[
	g(x_0)
	=
	\int_{\mathcal B}
	G_{\mathcal B}(x_0,y)H(y)\,dy
	\]
	for almost every \(x_0\in\mathcal B\). Consequently,
	\[
	|g(x_0)|
	\leq
	C\int_{\mathcal B}
	\frac{|H(y)|}{|x_0-y|^{n-2}}\,dy.
	\]
	Since $x_0\in\mathcal B$, Lemma~\ref{lem:FR-fixed-scale-H} implies
	\begin{equation}
		\|g\|_{L^\infty(\mathcal B)}
		\leq C.
		\label{eq:FR-g-Linfty}
	\end{equation}

	We next prove the weighted estimate for \(\nabla g\). Fix
	\(x_0\in B_{R_*}(z)\). Choose
	\[
	\xi\in C_c^\infty(B_{2R_*}(x_0)),
	\quad
	0\leq\xi\leq1,
	\quad
	\xi\equiv1\text{ on }B_{R_*}(x_0),
	\]
	such that
	\[
	|\nabla\xi|\leq CR_*^{-1},
	\quad
	|D^2\xi|\leq CR_*^{-2}.
	\]
	For \(h>0\), define
	\[
	\Phi_h(x):=(r^2+h^2)^{(2-n)/2},
	\quad
	r=|x-x_0|.
	\]
	A direct calculation gives
	\begin{equation}
		-\Delta\Phi_h
		=
		n(n-2)h^2(r^2+h^2)^{-(n+2)/2}
		\geq0.
		\label{eq:FR-Phi-laplacian}
	\end{equation}

	Using \(\xi^2\Phi_h g\) as a test function in
	\eqref{eq:FR-local-g}, we obtain
	\begin{equation}
		\int_{\mathcal B}
		|\nabla g|^2\xi^2\Phi_h\,dx
		=
		\int_{\mathcal B}
		Hg\xi^2\Phi_h\,dx
		+\frac12
		\int_{\mathcal B}
		g^2\Delta(\xi^2\Phi_h)\,dx.
		\label{eq:FR-g-weighted-test}
	\end{equation}
	Since
	\[
	\Delta(\xi^2\Phi_h)
	=
	\xi^2\Delta\Phi_h
	+2\nabla(\xi^2)\cdot\nabla\Phi_h
	+\Phi_h\Delta(\xi^2),
	\]
	and the first term on the right-hand side is nonpositive, we obtain
	\begin{align}
		\int_{\mathcal B}
		|\nabla g|^2\xi^2\Phi_h\,dx
		\leq
		\int_{\mathcal B}
		|H|\,|g|\,\xi^2\Phi_h\,dx
		+
		\left|
		\int_{\mathcal B}
		g^2\nabla(\xi^2)\cdot\nabla\Phi_h\,dx
		\right|
		+
		\frac12
		\left|
		\int_{\mathcal B}
		g^2\Phi_h\Delta(\xi^2)\,dx
		\right|.
		\label{eq:FR-g-weighted-estimate}
	\end{align}

	By \eqref{eq:FR-g-Linfty} and
	\eqref{eq:FR-fixed-scale-H},
	\[
	\begin{aligned}
		\int_{\mathcal B}
		|H|\,|g|\,\xi^2\Phi_h\,dx
		\leq
		\|g\|_{L^\infty(\mathcal B)}
		\int_{\mathcal B}|H|\Phi_h\,dx
		\leq C.
	\end{aligned}
	\]
	Moreover,
	\[
	\operatorname{supp}\nabla\xi
	\cup
	\operatorname{supp}\Delta\xi
	\subset
	B_{2R_*}(x_0)\setminus B_{R_*}(x_0).
	\]
	Using \eqref{eq:FR-g-Linfty}, we obtain
	\[
	\begin{aligned}
		\left|
		\int_{\mathcal B}
		g^2\nabla(\xi^2)\cdot\nabla\Phi_h\,dx
		\right|
		+
		\left|
		\int_{\mathcal B}
		g^2\Phi_h\Delta(\xi^2)\,dx
		\right|\leq
		C(R_*)\|g\|_{L^\infty(\mathcal B)}^2
		\leq C.
	\end{aligned}
	\]
	Consequently,
	\[
	\int_{\mathcal B}
	|\nabla g|^2\xi^2\Phi_h\,dx
	\leq C,
	\]
	uniformly in \(h\) and \(x_0\). Letting \(h\downarrow0\) and using
	Fatou's lemma gives
	\[
	\int_{B_{R_*}(x_0)}
	\frac{|\nabla g(x)|^2}{|x-x_0|^{n-2}}\,dx
	\leq C.
	\]
	Taking the supremum over \(x_0\in B_{R_*}(z)\), we conclude that
	\begin{equation}
		\|g\|_{L^\infty(B_{2R_*}(z))}
		+
		\sup_{x_0\in B_{R_*}(z)}
		\int_{B_{R_*}(x_0)}
		\frac{|\nabla g(x)|^2}{|x-x_0|^{n-2}}\,dx
		\leq C.
		\label{eq:FR-g-potential}
	\end{equation}

	{\it Step 2. Regularized weights.}
		From now on, fix
		\(
		x_0\in B_{R_*}(z),
		\quad
		0<R<\frac{R_*}{4},
		\) \quad
		and write
		\[
		T_R(x_0):=B_{2R}(x_0)\setminus B_R(x_0), \quad r=|x-x_0|.
		\]
		Choose
		\[
		\eta\in C_c^\infty(B_{2R}(x_0)),
		\quad
		0\leq\eta\leq1,
		\quad
		\eta\equiv1\quad\text{on }B_R(x_0),
		\]
		such that
		\begin{equation}
			|\nabla\eta|\leq\frac{C}{R},
			\quad
			|D^2\eta|\leq\frac{C}{R^2}.
			\label{eq:FR-cutoff-bounds}
		\end{equation}
		For \(0<h<R\), define
		\[
		w_h(x):=(r^2+h^2)^{-(n-4)/2}.
		\]
		A direct calculation gives
		\begin{equation}
		\nabla w_h
		=
		-(n-4)
		\frac{x-x_0}{(r^2+h^2)^{(n-2)/2}}
			\label{eq:FR-weight-gradient}
		\end{equation}
		and
		\begin{equation}
			-\Delta w_h
			=
			(n-4)
			\frac{2r^2+nh^2}{(r^2+h^2)^{n/2}}.
			\label{eq:FR-weight-laplacian}
		\end{equation}
		In particular,
		\begin{equation}
			2(n-4)(r^2+h^2)^{-(n-2)/2}
			\leq-\Delta w_h
			\leq
			n(n-4)(r^2+h^2)^{-(n-2)/2}.
			\label{eq:FR-weight-laplacian-bounds}
		\end{equation}
		We also have
		\begin{equation}
			|\nabla w_h|
			\leq
			C(r^2+h^2)^{-(n-3)/2},
			\quad
			|D^2w_h|
			\leq
			C(r^2+h^2)^{-(n-2)/2},
			\label{eq:FR-weight-derivative-bounds}
		\end{equation}
		and
		\begin{equation}
			\frac{|\nabla w_h|^2}{w_h}
			\leq
			C(r^2+h^2)^{-(n-2)/2}.
			\label{eq:FR-weight-square}
		\end{equation}

		{\it Step 3. The weighted velocity estimate.}
		We use $
		\varphi=\eta^2w_h
		$
		in \eqref{eq:FR-local-NS-inequality}. This gives
		\begin{equation}
			\int_{B_{2R}(x_0)} |\nabla u|^2\eta^2w_h\,dx
			+
			\int_{B_{2R}(x_0)} \partial_j u_i\,u_i
			\partial_j(\eta^2w_h)\,dx
			\leq
			\int_{B_{2R}(x_0)} H\,u\cdot\nabla(\eta^2w_h)\,dx.
			\label{eq:FR-velocity-test-1}
		\end{equation}

		Integration by parts yields
		\[
		\int_{B_{2R}(x_0)} \partial_j u_i\,u_i
		\partial_j(\eta^2w_h)\,dx
		=
		-\frac12
		\int_{B_{2R}(x_0)} |u|^2\Delta(\eta^2w_h)\,dx.
		\]
		Consequently,
		\begin{equation}
			\int_{B_{2R}(x_0)} |\nabla u|^2\eta^2w_h\,dx
			-
			\frac12\int_{B_{2R}(x_0)} |u|^2\Delta(\eta^2w_h)\,dx
			\leq
			\int_{B_{2R}(x_0)} H\,u\cdot\nabla(\eta^2w_h)\,dx.
			\label{eq:FR-velocity-test-2}
		\end{equation}

		We expand
		\[
		-\Delta(\eta^2w_h)
		=
		-\eta^2\Delta w_h
		-2\nabla(\eta^2)\cdot\nabla w_h
		-w_h\Delta(\eta^2).
		\]
		By \eqref{eq:FR-weight-laplacian-bounds},
		\begin{equation}
			-\frac12
			\int_{B_{2R}(x_0)} |u|^2\eta^2\Delta w_h\,dx
			\geq
			c_n
			\int_{B_{2R}(x_0)} |u|^2\eta^2
			(r^2+h^2)^{-(n-2)/2}\,dx.
			\label{eq:FR-u-positive-term}
		\end{equation}
		Using
		\eqref{eq:FR-cutoff-bounds} and
		\eqref{eq:FR-weight-derivative-bounds}, we find
		\begin{equation}
			\begin{aligned}
				\left|
				\int_{B_{2R}(x_0)} |u|^2\nabla(\eta^2)\cdot\nabla w_h\,dx
				\right|
				+
				\left|
				\int_{B_{2R}(x_0)} |u|^2w_h\Delta(\eta^2)\,dx
				\right|
				\leq
				C\int_{T_R(x_0)}
				\frac{|u|^2}{(r^2+h^2)^{(n-2)/2}}\,dx.
			\end{aligned}
			\label{eq:FR-u-annular-errors}
		\end{equation}

		Using \(H=-\Delta g\), the right-hand side of
		\eqref{eq:FR-velocity-test-2} becomes
		\[
		I_h
		:=
		-\int_{B_{2R}(x_0)} \Delta g\,
		u_i\partial_i(\eta^2w_h)\,dx.
		\]
		Integration by parts gives
		\begin{equation}
			\begin{aligned}
				I_h
				=
				\int_{B_{2R}(x_0)} \partial_k g\,\partial_k u_i\,
				\partial_i(\eta^2w_h)\,dx +
				\int_{B_{2R}(x_0)} \partial_k g\,u_i\,
				\partial_{ki}(\eta^2w_h)\,dx
				=:I_{h,1}+I_{h,2}.
			\end{aligned}
			\label{eq:FR-Ih-decomposition}
		\end{equation}

		By Young's inequality and
		\eqref{eq:FR-weight-square},
		\begin{equation}
			\begin{aligned}
				|I_{h,1}|
				\leq
				\varepsilon
				\int_{B_{2R}(x_0)} |\nabla u|^2\eta^2w_h\,dx+
				C_\varepsilon
				\int_{B_{2R}(x_0)}
				\frac{|\nabla g|^2}
				{(r^2+h^2)^{(n-2)/2}}\,dx.
			\end{aligned}
			\label{eq:FR-Ih1}
		\end{equation}

		To estimate \(I_{h,2}\), we expand
		\[
		D^2(\eta^2w_h)
		=
		\eta^2D^2w_h
		+
		\nabla(\eta^2)\otimes\nabla w_h
		+\nabla w_h\otimes\nabla(\eta^2)
		+
		w_hD^2(\eta^2).
		\]
		Using \eqref{eq:FR-cutoff-bounds},
		\eqref{eq:FR-weight-derivative-bounds}, and Young's inequality, we
		obtain
		\begin{equation}
			\begin{aligned}
				|I_{h,2}|
				\leq & \varepsilon
				\int_{B_{2R}(x_0)} |u|^2\eta^2
				(r^2+h^2)^{-(n-2)/2}\,dx
				+
				C_\varepsilon
				\int_{B_{2R}(x_0)}
				\frac{|\nabla g|^2}
				{(r^2+h^2)^{(n-2)/2}}\,dx\\
		&\quad	+
				C_\varepsilon
				\int_{T_R(x_0)}
				\frac{|u|^2+|\nabla g|^2}
				{(r^2+h^2)^{(n-2)/2}}\,dx.
			\end{aligned}
			\label{eq:FR-Ih2}
		\end{equation}

		Choosing \(\varepsilon>0\) sufficiently small and absorbing the
		corresponding terms into the left-hand side of
		\eqref{eq:FR-velocity-test-2}, we conclude that
		\begin{equation}
			\begin{aligned}
				&\int_{B_R(x_0)}
				\frac{|\nabla u|^2}
				{(r^2+h^2)^{(n-4)/2}}\,dx
				+
				\int_{B_R(x_0)}
				\frac{|u|^2}
				{(r^2+h^2)^{(n-2)/2}}\,dx
				\\
				&\quad\leq
				K_1\int_{T_R(x_0)}
				\frac{|u|^2}
				{(r^2+h^2)^{(n-2)/2}}\,dx
				+
				K_1\int_{B_{2R}(x_0)}
				\frac{|\nabla g|^2}
				{(r^2+h^2)^{(n-2)/2}}\,dx.
			\end{aligned}
			\label{eq:FR-velocity-weighted-h}
		\end{equation}
		The right-hand side is uniformly bounded as $h\downarrow0$:
		the annular weight is bounded, and the weighted $\nabla g$ integral
		is controlled by \eqref{eq:FR-g-potential}, since $2R<R_*$.
		Monotone convergence therefore yields
		\begin{equation}
			\begin{aligned}
				&\int_{B_R(x_0)}
				\frac{|\nabla u|^2}{r^{n-4}}\,dx
				+
				\int_{B_R(x_0)}
				\frac{|u|^2}{r^{n-2}}\,dx
				\\
				&\quad\leq
				K_1\int_{T_R(x_0)}
				\frac{|u|^2}{r^{n-2}}\,dx
				+
				K_1\int_{B_{2R}(x_0)}
				\frac{|\nabla g|^2}{r^{n-2}}\,dx.
			\end{aligned}
			\label{eq:FR-velocity-weighted}
		\end{equation}
		Note that $K_1$ is independent of $R$ and $x_0$.

	{\it Step 4. The weighted estimate for the auxiliary function.}
		Set
		\[
		g_R:=\fint_{T_R(x_0)}g(x)\,dx
		\]
		and define
		\[
		v_h(x):=(r^2+h^2)^{-(n-2)/2}.
		\]
		We use
		\[
		\psi_h=(g-g_R)\eta^2v_h
		\]
		as a test function in \eqref{eq:FR-local-g}. We obtain
		\begin{equation}
			\begin{aligned}
				\int_{B_{2R}(x_0)} |\nabla g|^2\eta^2v_h\,dx
				=&
				-2\int_{B_{2R}(x_0)}
				(g-g_R)\eta v_h\nabla g\cdot\nabla\eta\,dx
				\\
				&-
				\int_{B_{2R}(x_0)}
				(g-g_R)\eta^2\nabla g\cdot\nabla v_h\,dx+
				\int_{B_{2R}(x_0)}
				H(g-g_R)\eta^2v_h\,dx.
			\end{aligned}
			\label{eq:FR-g-test}
		\end{equation}

		For the middle term, integration by parts gives
		\begin{equation}
			\begin{aligned}
				-\int_{B_{2R}(x_0)}
				(g-g_R)\eta^2\nabla g\cdot\nabla v_h\,dx
				=
				\frac12\int_{B_{2R}(x_0)}
				|g-g_R|^2\eta^2\Delta v_h\,dx
				+
				\int_{B_{2R}(x_0)}
				|g-g_R|^2\eta\nabla\eta\cdot\nabla v_h\,dx.
			\end{aligned}
			\label{eq:FR-g-middle-term}
		\end{equation}
		A direct computation gives
		\[
		\Delta v_h
		=
		-n(n-2)h^2(r^2+h^2)^{-(n+2)/2}
		\leq0.
		\]
		The first term on the right-hand side of
		\eqref{eq:FR-g-middle-term} is therefore nonpositive and may be
		discarded.

		Using Young's inequality together with the Poincar\'e inequality on
		the annulus,
		\[
		\int_{T_R(x_0)}|g-g_R|^2\,dx
		\leq
		CR^2\int_{T_R(x_0)}|\nabla g|^2\,dx,
		\]
		we obtain
		\begin{equation}
			\begin{aligned}
				&\left|
				-2\int_{B_{2R}(x_0)}
				(g-g_R)\eta v_h\nabla g\cdot\nabla\eta\,dx
				\right|
			+
				\left|
				\int_{B_{2R}(x_0)}
				|g-g_R|^2\eta\nabla\eta\cdot\nabla v_h\,dx
				\right|
				\\
				&\quad\leq
				\frac14
				\int_{B_{2R}(x_0)} |\nabla g|^2\eta^2v_h\,dx
				+
				C\int_{T_R(x_0)}
				\frac{|\nabla g|^2}
				{(r^2+h^2)^{(n-2)/2}}\,dx.
			\end{aligned}
			\label{eq:FR-g-annular}
		\end{equation}
		The last term in \eqref{eq:FR-g-test} is estimated using
		\eqref{eq:FR-g-Linfty} and Lemma~\ref{lem:FR-fixed-scale-H}:
		\begin{equation}
			\left|\int_{B_{2R}(x_0)}
			H(g-g_R)\eta^2v_h\,dx \right|
			\leq
			C \|g\|_{L^\infty(B_{3R_*/2 }(z))} \int_{B_{2R}(x_0)}
			|H| v_h\,dx \leq C\int_{B_{2R}(x_0)}
			|H| v_h\,dx.
			\label{eq:FR-g-H-term}
		\end{equation}
		It follows from
		\eqref{eq:FR-g-test}-\eqref{eq:FR-g-H-term} that
		\begin{equation}
			\begin{aligned}
				\int_{B_R(x_0)}
				\frac{|\nabla g|^2}
				{(r^2+h^2)^{(n-2)/2}}\,dx
				\leq
				K_2\int_{B_{2R}(x_0)}
				\frac{|H|}
				{(r^2+h^2)^{(n-2)/2}}\,dx
				+
				K_2\int_{T_R(x_0)}
				\frac{|\nabla g|^2}
				{(r^2+h^2)^{(n-2)/2}}\,dx.
			\end{aligned}
			\label{eq:FR-g-weighted-h}
		\end{equation}
		Letting \(h\downarrow0\) and using monotone convergence on all
		nonnegative weighted terms, we obtain
		\begin{equation}
			\begin{aligned}
				\int_{B_R(x_0)}
				\frac{|\nabla g|^2}{r^{n-2}}\,dx
				\leq
				K_2\int_{B_{2R}(x_0)}
				\frac{|H|}{r^{n-2}}\,dx
				+
				K_2\int_{T_R(x_0)}
				\frac{|\nabla g|^2}{r^{n-2}}\,dx.
			\end{aligned}
			\label{eq:FR-g-weighted}
		\end{equation}
		Note that $K_2$ is independent of $R$ and $x_0$.

		{\it Step 5. The weighted head-pressure estimate.}
	Let \(\psi = \eta^2 w_h.\)  We use $\nabla\psi$ as a test function for
	the Navier-Stokes equations.  A direct computation using the formulas for
	$D^2w_h$ and $\Delta w_h$ gives
 \begin{equation}
 	\begin{aligned}
 		& -2(n-4) \int_{B_{2R}(x_0)} \frac{H \eta^2}{(r^2+h^2)^{(n-2)/2}}\, dx + (n-4)(n-2) \int_{B_{2R}(x_0)} \frac{|u\cdot(x-x_0)|^2 \eta^2}{(r^2+h^2)^{n/2}} \, dx\\
 		& \quad  - (n-4)(n-2) \int_{B_{2R}(x_0)} \frac{p\eta^2 h^2}{(r^2+h^2)^{n/2}}\, dx \\
 		& \quad = 4(n-4) \int_{B_{2R}(x_0)} \frac{p \eta \nabla \eta \cdot (x-x_0) }{(r^2+h^2)^{(n-2)/2}} \, dx  - 2\int_{B_{2R}(x_0)} \frac{p|\nabla \eta|^2}{(r^2+h^2)^{(n-4)/2}}  \, dx \\
 		& \quad - 2\int_{B_{2R}(x_0)}
 		\frac{p\eta \Delta \eta }{(r^2+h^2)^{(n-4)/2}}\, dx
 	 -\int_{B_{2R}(x_0)} u_i u_j \left[\frac{2\partial_i \eta\partial_j \eta + 2\eta \partial_{ij}\eta }{(r^2+h^2)^{(n-4)/2}} \right]\, dx \\
 		& \quad + 4(n-4)\int_{B_{2R}(x_0)} \frac{u\cdot(x-x_0)  u_j   \partial_j \eta \eta }{ (r^2+h^2)^{(n-2)/2}}\, dx \\
 		&\quad =: J_1 +\cdots + J_5.
 			\end{aligned}
 \end{equation}
	Using $|H|=2H_{+}-H$ and $p=H-\frac12|u|^2$, we obtain the following
	identity.  
	\begin{equation}
		\begin{aligned}
		&2(n-4) \int_{B_{2R}(x_0)} \frac{|H| \eta^2}{(r^2+h^2)^{(n-2)/2}}\, dx + (n-4)(n-2) \int_{B_{2R}(x_0)} \frac{|u\cdot(x-x_0)|^2 \eta^2}{(r^2+h^2)^{n/2}} \, dx\\
		&\quad + \frac{(n-4)(n-2)}{2} \int_{B_{2R}(x_0)} \frac{|u|^2 \eta^2 h^2 }{(r^2 +h^2)^{n/2}} \, dx  + (n-4)(n-2) \int_{B_{2R}(x_0)}\frac{|H| \eta^2 h^2}{(r^2+ h^2)^{n/2}}\, dx 		\\
		&\quad =J_1 + \cdots + J_5 + 4(n-4) \int_{B_{2R}(x_0)} \frac{H_{+} \eta^2}{(r^2+h^2)^{(n-2)/2}} \, dx \\
		&\quad + 2(n-4)(n-2) \int_{B_{2R}(x_0)} \frac{H_{+} \eta^2 h^2}{(r^2 + h^2)^{n/2}}\, dx =: J_1 +\cdots +J_7.
		\end{aligned}
	\end{equation}
	Here $J_1,\ldots,J_5$ denote the five cutoff terms from the preceding
	identity, while $J_6$ and $J_7$ are the two terms containing $H_+$.
	For the integrals on the right-hand side we have
	\begin{equation*}
		|J_1|+|J_2|+|J_3| \leq C\int_{T_R(x_0)} \frac{|p|}{r^{n-2}} \, dx \leq C \int_{T_R(x_0)} \frac{|H|}{r^{n-2}} \, dx + C\int_{T_R(x_0)} \frac{|u|^2}{r^{n-2}}\, dx .
	\end{equation*}
	\begin{equation*}
		|J_4|+|J_5| \leq C \int_{T_R(x_0)} \frac{|u|^2}{r^{n-2}}\, dx .
	\end{equation*}
	\begin{equation*}
		|J_6| +|J_7| \leq C \|H_{+}\|_{L^q(B_{2R}(x_0))} R^{2-\frac{n}{q}} \leq C R^{2-\frac{n}{q}} .
	\end{equation*}
	Consequently,
	\begin{equation}\label{eq:FR-H-weighted-1}
		\begin{aligned}
		&	\int_{B_R(x_0)} \frac{|H|}{(r^2 + h^2)^{(n-2)/2}} \, dx +
			\int_{B_R(x_0)} \frac{|u\cdot (x-x_0)|^2}{(r^2+h^2)^{n/2}} \, dx \\
			&\quad \leq K_3 \int_{T_R(x_0)} \frac{|H|}{r^{n-2}} \, dx + K_3\int_{T_R(x_0)} \frac{|u|^2}{r^{n-2}}\, dx + K_3  R^{2-\frac{n}{q}}.
		\end{aligned}
	\end{equation}
	Letting \(h\downarrow0\), the left-hand side converges monotonically,
	while the terms on the annulus $T_R(x_0)$ are handled by dominated
	convergence (there $r\simeq R$).  We obtain
\begin{equation}
	\begin{aligned}
		&\int_{B_R(x_0)}
		\frac{|H|}{r^{n-2}}\,dx
		+
		\int_{B_R(x_0)}
		\frac{|u\cdot (x-x_0)|^2}{r^{n}}\,dx
		\\
		&\quad\leq
		K_3\int_{T_R(x_0)}
		\frac{|H| }{r^{n-2}}\,dx
		+ K_3 \int_{T_R(x_0)}
		\frac{|u|^2}{r^{n-2}}\,dx +  K_3  R^{2-\frac{n}{q}}.
	\end{aligned}
	\label{eq:FR-H-weighted}
\end{equation}
Note that $K_3$ is independent of $R$ and $x_0$.

	{\it Step 6. The hole-filling inequality.}
	Set
	\[
	\gamma:=2-\frac{n}{q}>0.
	\]
	Introduce the nonnegative quantity
	\begin{equation}
		\begin{aligned}
			\mathcal E(R):={}&
			\int_{B_R(x_0)}
			\frac{|\nabla u|^2}{r^{n-4}}\,dx
			+
			\int_{B_R(x_0)}
			\frac{|u|^2}{r^{n-2}}\,dx
			+
			\int_{B_R(x_0)}
			\frac{|\nabla g|^2}{r^{n-2}}\,dx
			\\
			&+
			\int_{B_R(x_0)}
			\frac{|H|}{r^{n-2}}\,dx
			+
			\int_{B_R(x_0)}
			\frac{|u\cdot(x-x_0)|^2}{r^n}\,dx.
		\end{aligned}
		\label{eq:FR-total-energy}
	\end{equation}
	Before taking differences, note that $\mathcal E(2R)<\infty$. 
	Indeed, \eqref{eq:FR-velocity-weighted} at the fixed radius
	$R_*/8$, together with \eqref{eq:FR-g-potential}, gives finite
	weighted velocity integrals near $x_0$; outside that ball the weights
	are bounded and $u\in W^{1,2}$.  The remaining terms are finite by
	\eqref{eq:FR-fixed-scale-H}, \eqref{eq:FR-g-potential}, and
	$|u\cdot(x-x_0)|^2/r^n\leq |u|^2/r^{n-2}$.

	For clarity, split
	\[
	 \mathcal E(R)=U(R)+G_0(R)+P(R),
	\]
	where $U$ is the sum of the first two terms in
	\eqref{eq:FR-total-energy}, $G_0$ is the third term, and $P$ is the sum
	of the last two terms.  Since every integrand is nonnegative, the three
	weighted estimates imply
	\begin{align}
	 U(R)&\le K_1\bigl(U(2R)-U(R)\bigr)+K_1G_0(2R),
	 \label{eq:FR-hole-U}\\
	 G_0(R)&\le K_2P(2R)+K_2\bigl(G_0(2R)-G_0(R)\bigr),
	 \label{eq:FR-hole-G}\\
	 P(R)&\le K_3\bigl(P(2R)-P(R)\bigr)
	 +K_3\bigl(U(2R)-U(R)\bigr)+K_3R^\gamma.
	 \label{eq:FR-hole-P}
	\end{align}
	Here, for example, the annular $|u|^2$ term in
	\eqref{eq:FR-velocity-weighted} is bounded by
	$U(2R)-U(R)$, and the $|H|$ term over $B_{2R}$ in
	\eqref{eq:FR-g-weighted} is bounded by $P(2R)$.
	Now multiply \eqref{eq:FR-hole-G} by $K_1+1$ and
	\eqref{eq:FR-hole-P} by
	\[
	 (K_1+1)K_2+1,
	\]
	and add the resulting inequalities to \eqref{eq:FR-hole-U}.  On the
	right-hand side write
	$G_0(2R)=G_0(R)+[G_0(2R)-G_0(R)]$ and
	$P(2R)=P(R)+[P(2R)-P(R)]$.  Moving the resulting $K_1G_0(R)$ and
	$(K_1+1)K_2P(R)$ terms to the left leaves exactly
	\[
	 U(R)+G_0(R)+P(R)=\mathcal E(R).
	\]
	All remaining terms are nonnegative multiples of components of
	$\mathcal E(2R)-\mathcal E(R)$, together with the error $CR^\gamma$.
	Thus, after enlarging a constant $K_4\ge1$ independent of $x_0$ and $R$,
	we obtain
	\begin{equation}
	 \mathcal E(R)
	 \le K_4\bigl(\mathcal E(2R)-\mathcal E(R)\bigr)+K_4R^\gamma.
	 \label{eq:FR-hole-filling-1}
	\end{equation}
	Hence
	\begin{equation}
		\mathcal E(R)
		\leq
		\theta\mathcal E(2R)
		+
		CR^\gamma,
		\quad
		\theta:=\frac{K_4}{1+K_4}\in(0,1).
		\label{eq:FR-hole-filling-2}
	\end{equation}

	We first obtain a uniform bound at a fixed intermediate scale.
	Set
	\[
	R_0:=\frac{R_*}{8}.
	\]
	We claim that
	\begin{equation}
		\sup_{\substack{x_0\in B_{R_*}(z)\\
				R_0\leq S<2R_0}}
		\mathcal E(S)\leq C.
		\label{eq:FR-terminal-scale}
	\end{equation}
	Indeed, since \(S<2R_0=R_*/4\), the estimate
	\eqref{eq:FR-velocity-weighted} applies with \(R=S\). Moreover,
	on the annulus \(T_S(x_0)\) we have \(r\geq S\geq R_0\), and hence
	\[
	\int_{T_S(x_0)}
	\frac{|u|^2}{r^{n-2}}\,dx
	\leq
	R_0^{2-n}
	\|u\|_{L^2(B_{4R_*}(z))}^2
	\leq C.
	\]
	By \eqref{eq:FR-g-potential},
	\[
	\int_{B_{2S}(x_0)}
	\frac{|\nabla g|^2}{r^{n-2}}\,dx
	\leq C.
	\]
	Therefore \eqref{eq:FR-velocity-weighted} gives
	\[
	\int_{B_S(x_0)}
	\frac{|\nabla u|^2}{r^{n-4}}\,dx
	+
	\int_{B_S(x_0)}
	\frac{|u|^2}{r^{n-2}}\,dx
	\leq C.
	\]
	The remaining velocity component of $\mathcal E(S)$ is controlled
	directly by the second term above, since
	\[
	 \frac{|u\cdot(x-x_0)|^2}{r^n}
	 \le \frac{|u|^2r^2}{r^n}
	 =\frac{|u|^2}{r^{n-2}}.
	\]
	Hence
	\[
	\int_{B_S(x_0)}
	\frac{|u\cdot (x-x_0)|^2}{r^{n}}\,dx
	\le C.
	\]
	Furthermore, Lemma~\ref{lem:FR-fixed-scale-H} and
	\eqref{eq:FR-g-potential} imply
	\[
	\int_{B_S(x_0)}
	\frac{|H|}{r^{n-2}}\,dx
	+
	\int_{B_S(x_0)}
	\frac{|\nabla g|^2}{r^{n-2}}\,dx
	\leq C.
	\]
	Thus \eqref{eq:FR-terminal-scale} follows.

	Choose
	\[
	0<\beta<
	\min\left\{
	\gamma,-\log_2\theta,1
	\right\}.
	\]
	Let \(0<R<R_0\), and choose the smallest integer \(m\geq1\)
	such that
	\[
	2^mR\geq R_0.
	\]
	Then
	\[
	R_0\leq 2^mR<2R_0=\frac{R_*}{4}.
	\]
	Iterating \eqref{eq:FR-hole-filling-2} at the radii
	\[
	R,\ 2R,\ \ldots,\ 2^{m-1}R
	\]
	gives
	\begin{equation}
		\mathcal E(R)
		\leq
		\theta^m\mathcal E(2^mR)
		+
		C\sum_{j=0}^{m-1}
		\theta^j(2^jR)^\gamma.
		\label{eq:FR-hole-filling-iteration}
	\end{equation}
	Since
	\[
	\beta<-\log_2\theta,
	\]
	we have
	\[
	\theta\,2^\beta<1.
	\]
	Also, by the minimality of \(m\),
	\[
	2^mR\geq R_0,
	\]
	and therefore
	\[
	\theta^m
	\leq 2^{-m\beta}
	\leq
	\left(\frac{R}{R_0}\right)^\beta.
	\]
	Using \eqref{eq:FR-terminal-scale}, we conclude that
	\[
	\theta^m\mathcal E(2^mR)
	\leq
	C\left(\frac{R}{R_0}\right)^\beta.
	\]
	For the error term, since \(2^jR<2R_0\) and
	\(\gamma>\beta\),
	\[
	(2^jR)^\gamma
	=
	(2^jR)^\beta(2^jR)^{\gamma-\beta}
	\leq
	(2R_0)^{\gamma-\beta}
	R^\beta 2^{j\beta}.
	\]
	Hence
	\[
	\sum_{j=0}^{m-1}
	\theta^j(2^jR)^\gamma
	\leq
	C R^\beta
	\sum_{j=0}^{m-1}
	(\theta2^\beta)^j
	\leq
	CR^\beta.
	\]
	It follows from \eqref{eq:FR-hole-filling-iteration} that
	\[
	\mathcal E(R)\leq CR^\beta,
	\quad
	0<R<R_0.
	\]
	For \(R_0\leq R<R_*/4\), the same fixed-scale estimates used in
	\eqref{eq:FR-terminal-scale} give
	\[
	\mathcal E(R)\leq C.
	\]
	Since \(R\geq R_0\), after enlarging the constant if necessary,
	\[
	\mathcal E(R)\leq CR^\beta.
	\]
	Consequently,
	\[
	\mathcal E(R)\leq CR^\beta,
	\quad
	0<R<\frac{R_*}{4},
	\]
	uniformly for \(x_0\in B_{R_*}(z)\).
	In particular,
	\[
	\int_{B_R(x_0)}
	\frac{|\nabla u(x)|^2}
	{|x-x_0|^{n-4}}\,dx
	\leq CR^\beta,
	\]
	which is exactly
	\eqref{eq:FR-interior-weighted-decay}.
	This completes the proof of
	Proposition~\ref{prop:FR-interior-weighted}.
\end{proof}

        \medskip

{\bf Acknowledgements.}
The research of Gui is supported by  NSFC Key Program (Grant No. 12531010),  University of Macau research grants CPG2024-00016-FST, CPG2025-00032-FST, CPG2026-00027-FST, SRG2023-00011-FST, MYRG-GRG2023-00139-FST-UMDF, UMDF Professorial Fellowship of Mathematics, Macao SAR FDCT 0003/2023/RIA1 and  Macao SAR FDCT 0024/2023/RIB1.
The research of Wang is partially supported by NSFC Grant 12671279, the Natural Science Foundation of Jiangsu Province (Grant~BK20240147), and the Jiangsu Provincial Scientific Research Center of Applied Mathematics (No.~BK20233002).
The research of Xie is partially supported by NSFC Grants 12571238 and 12426203.

\end{document}